\pdfoutput=1

\documentclass{amsart}

\usepackage{
 amsmath,
 amssymb,
 enumerate,
 fancyhdr,
 fullpage,
 mathrsfs,
 mathtools,
 thmtools,
 tikz-cd}

\usepackage[pagebackref, hypertexnames=false]{hyperref}
\usepackage[capitalize]{cleveref}

 \renewcommand{\AA}{\mathbf{A}}
 \newcommand{\RR}  {\mathbf{R}}
 \newcommand{\CC}  {\mathbf{C}}
 \newcommand{\GG}  {\mathbf{G}}
 \newcommand{\QQ}  {\mathbf{Q}}
 \renewcommand{\SS}{\mathbf{S}}
 \newcommand{\ZZ}  {\mathbf{Z}}

 \newcommand{\Ql}  {\QQ_\ell}
 \newcommand{\Qlt} {\Ql^\times}
 \newcommand{\Qp}  {\QQ_p}
 \newcommand{\Zl}  {\ZZ_\ell}
 \newcommand{\Zlt} {\Zl^\times}
 \newcommand{\Zp}  {\ZZ_p}
 \newcommand{\Af}  {\AA_\mathrm{f}}
 \newcommand{\pif} {\pi_\mathrm{f}}

 \newenvironment{smatrix}{\left( \begin{smallmatrix} } {\end{smallmatrix} \right) }
 
 \newcommand{\stbt}[4]{\begin{smatrix}#1 & #2 \\ #3 & #4\end{smatrix}}
 \newcommand{\sdth}[3]{\begin{smatrix}#1\\ &#2 \\ &&#3\end{smatrix}}

 \newcommand{\fm}{\mathfrak{m}}
 \newcommand{\fn}{\mathfrak{n}}
 \newcommand{\fp}{\mathfrak{p}}
 \newcommand{\fP}{\mathfrak{P}}

 \newcommand{\cO}{\mathcal{O}}
 \newcommand{\cH}{\mathcal{H}}
 \newcommand{\cI}{\mathcal{I}}
 \newcommand{\cP}{\mathcal{P}}
 \newcommand{\cR}{\mathcal{R}}
 \newcommand{\cS}{\mathcal{S}}
 \newcommand{\cW}{\mathcal{W}}
 \newcommand{\cX}{\mathcal{X}}
 \newcommand{\cY}{\mathcal{Y}}
 \newcommand{\cZ}{\mathcal{Z}}

 \newcommand{\sD}{\mathscr{D}}
 \newcommand{\sG}{\mathscr{G}}
 \newcommand{\sH}{\mathscr{H}}
 \newcommand{\sV}{\mathscr{V}}
 \newcommand{\sX}{\mathscr{X}}

 \newcommand{\into}{\hookrightarrow}
 \newcommand{\isoarrow}{\stackrel{\sim}{\rightarrow}}
 \newcommand{\dd}{\, \mathrm{d}}
 \newcommand{\ddt}{\dd^\times \!}
 \newcommand{\bs}{\backslash}
 
 \renewcommand{\t}{{}^t}
 \newcommand{\UE}{{\mathcal{UE}}}

 \newcommand{\et}{\text{\textup{\'et}}}
 \newcommand{\mot}{\mathrm{mot}}
 \newcommand{\ord}{\mathrm{ord}}

 \DeclareMathOperator{\Anc}{Anc}
 \DeclareMathOperator{\Art}{Art}
 \DeclareMathOperator{\Aut}{Aut}
 \DeclareMathOperator{\Eis}{Eis}
 \DeclareMathOperator{\Gal}{Gal}
 \DeclareMathOperator{\GL}{GL}
 \DeclareMathOperator{\Hom}{Hom}
 \DeclareMathOperator{\Ih}{Ih}
 \DeclareMathOperator{\Iw}{Iw}
 \DeclareMathOperator{\Nm}{Nm}
 \DeclareMathOperator{\Res}{Res}
 \DeclareMathOperator{\Spec}{Spec}
 \DeclareMathOperator{\Sym}{Sym}

 \DeclareMathOperator{\br}{br}
 \DeclareMathOperator{\ch}{ch}
 \DeclareMathOperator{\diag}{diag}
 \DeclareMathOperator{\loc}{loc}
 \DeclareMathOperator{\mom}{mom}
 \DeclareMathOperator{\bc}{bc}
 \DeclareMathOperator{\pr}{pr}
 \DeclareMathOperator{\vol}{vol}

 \renewcommand{\le}{\leqslant}
 \renewcommand{\leq}{\leqslant}
 \renewcommand{\ge}{\geqslant}
 \renewcommand{\geq}{\geqslant}

 \makeatletter
 \@addtoreset{equation}{subsection}
 
 \makeatother

 \newtheorem{theorem}{Theorem}[subsection]
 \newtheorem{definition} [theorem]{Definition}
 \newtheorem{corollary}  [theorem]{Corollary}
 \newtheorem{lemma}      [theorem]{Lemma}
 \newtheorem{proposition}[theorem]{Proposition}

 \newtheorem{ltheorem}{Theorem}

 \declaretheorem[name=Proposition,sibling=definition,qed=\qedsymbol]{propqed}

 \theoremstyle{remark}
 \declaretheorem[name=Remark,sibling=theorem,qed={\lower-0.3ex\hbox{$\diamond$}}]{remark}
 \declaretheorem[name=Note,sibling=theorem,qed={\lower-0.3ex\hbox{$\diamond$}}]{note}
 \declaretheorem[name=Notation,unnumbered,qed={\lower-0.3ex\hbox{$\diamond$}}]{notation}

 \author{David Loeffler}
 \address[Loeffler]{
  Mathematics Institute \\
  University of Warwick \\
  Coventry CV4 7AL, UK}
 \email{d.a.loeffler@warwick.ac.uk}

 \author{Christopher Skinner}
 \address[Skinner]{
  Department of Mathematics\\
  Princeton University\\
  Princeton, NJ\\
  USA}
 \email{cmcls@princeton.edu}

 \author{Sarah Livia Zerbes}
 \address[Zerbes]{
  Department of Mathematics \\
  University College London \\
  London, UK}
 \email{s.zerbes@ucl.ac.uk}

\thanks{Supported by: Royal Society University Research Fellowship ``$L$-functions and Iwasawa theory'' (Loeffler); Simons Investigator Grant \#376203 from the  Simons Foundation and and NSF grant DMS-1501064 (Skinner); ERC Consolidator Grant ``Euler systems and the Birch--Swinnerton-Dyer conjecture'' (Zerbes).}
 \title[An Euler system for GU(2,\,1)]
 {An Euler system for GU(2,\,1)}
\begin{document}

\begin{abstract}
 We construct an Euler system associated to regular algebraic, essentially conjugate self-dual cuspidal automorphic representations of $\GL_3$ over imaginary quadratic fields, using the cohomology of Shimura varieties for $\operatorname{GU}(2, 1)$.
\end{abstract}
\maketitle

%\setcounter{tocdepth}{1}

%\tableofcontents

%\thispagestyle{fancyplain} %% for preprint footer, kill on submission!

\section{Introduction}

 \subsection{Overview of the results}

  Euler systems -- families of global cohomology classes satisfying norm-compatibility relations -- are among the most powerful tools available for studying the arithmetic of global Galois representations. In particular, most of the known cases of the Bloch--Kato conjecture, and of the Iwasawa main conjecture, use Euler systems as a fundamental ingredient in their proofs. However, Euler systems are correspondingly difficult to construct; in almost all known cases, the construction uses automorphic tools, relying on the motivic cohomology of Shimura varieties.

  Euler systems come in two flavours: \emph{full Euler systems}, in which we have classes over almost all of the ray class fields $E[\fm]$, where $E$ is some fixed number field; or \emph{anticyclotomic Euler systems}, where $E$ is a CM field, and we restrict to ring class fields (the anticyclotomic parts of ray class fields). Full Euler systems are the most powerful for applications, but correspondingly hardest to construct.

  In this paper, we'll construct a new example of a full Euler system, associated to Shimura varieties for the group $G = \operatorname{GU}(2, 1)$ (Picard modular surfaces). This construction has some novel features compared with previous constructions, such as the $\operatorname{GSp}_4$ case treated in \cite{LSZ17}. Firstly, the field $E$ (which is the reflex field of the Shimura datum for $G$) is not $\QQ$, but an imaginary quadratic field, and so an Euler system in this setting consists of classes over all of the abelian extensions of $E$ (most of which are not abelian over $\QQ$). Secondly, we introduce here a new strategy for proving norm-compatibility relations, based on cyclicity results for local Hecke algebras; this allows us to show that our classes are norm-compatible in the strongest possible sense, i.e.~as classes in motivic cohomology (whereas in \cite{LSZ17} we only proved norm relations for the images of Euler system classes in the \'etale realisation, after projecting to an appropriate Hecke eigenspace). Such cyclicity results for Hecke algebras are closely bound up with the theory of \emph{spherical varieties}, and we believe that this connection with spherical varieties should be a fruitful tool for studying Euler systems in many other contexts.

  \begin{ltheorem}
   Let $G = \operatorname{GU}(2, 1)$, $K_G$ an open compact subgroup of $G(\Af)$, and $\Sigma(K_G)$ the set of primes which ramify in $E$ or divide the level of $K_G$. Let $c > 1$ be an integer coprime to $6\Sigma(K_G)$; and let $\cR$ be the set of squarefree products $\fm$ of primes $w$ of $E$ coprime to $c \Sigma(K_G)$ with the following property: if $\ell = w \bar{w}$ is a split prime, then at most one of $w$ and $\bar{w}$ divides $\fm$. Let $0 \le r \le a, 0 \le s \le b$ be integers.

   Then there exists a family of motivic cohomology classes
   \[ {}_c \Xi_{\mot, \fm}^{[a,b,r,s]} \in H^3_{\mot}\left(Y_G(K_G) \times_{E} E[\fm], \sD^{a,b}\{r, s\}(2)\right) \]
   for all $\fm \in \cR$, where $E[\fm]$ is the ray class field modulo $\fm$, with the following properties:
   \begin{enumerate}
   \item If $\fm, \fn \in \cR$ with $\fm \mid \fn$, then
   \[ \operatorname{norm}_{E[\fm]}^{E[\fn]}\left({}_c \Xi_{\mot, \fm}^{[a,b,r,s]}\right) = \Big( \prod_{w \mid \tfrac{\fn}{\fm}} \cP_w'(\sigma_w^{-1}) \Big)\, {}_c\Xi_{\mot, \fm}^{[a,b,r,s]},
    \]
   where $\cP_w'(X)$ is a polynomial over the spherical Hecke algebra (which acts on each eigenspace as an Euler factor at $w$), and $\sigma_w \in \Gal(E[\fm] / E)$ is the arithmetic Frobenius at $w$.

   \item For any prime $\fp$ of $E$ not dividing $\Sigma(K_G) \Nm(\fm)$, the image of the class ${}_c \Xi_{\mot, \fm}^{[a,b,r,s]}$ under the $\fp$-adic \'etale realisation map is integral (i.e.~lies in the \'etale cohomology with $\cO_{E, \fp}$-coefficients).
   \end{enumerate}
  \end{ltheorem}

  We refer the reader to \S \ref{sect:shimura} for the definition of the Shimura variety $Y_G(K_G)$, and the relative Chow motive $\sD^{a,b}\{r, s\}$ over it. In the case $(a,b,r,s) = (0,0,0,0)$, this motive is simply the trivial motive $E(0)$, and our classes coincide with those considered in \cite{pollackshah18}; in particular, the main result of \emph{op.cit.}~shows that the images of these classes under the Deligne--Beilinson regulator map, paired with suitable real-analytic differential forms on $Y_G(K_G)(\CC)$, are related to the values $L'(\pi, 0)$ for cuspidal automorphic representations $\pi$ of $G(\AA)$. This shows that our motivic cohomology classes are non-zero in this trivial-coefficient case. (We expect that a complex regulator formula similar to \cite{pollackshah18} should also hold for more general coefficient systems, but we shall not treat this problem here.)

  After passing to a Shimura variety with Iwahori level structure at $p$, we can also obtain families of classes over all the fields $E[\fm p^t]$ for $t \ge 1$, satisfying a norm-compatibility in both $\fm$ and $t$; see \cref{thm:wildnorm} for the precise statement. Applying the \'etale regulator map and projecting to a cuspidal Hecke eigenspace, we obtain Euler systems in the conventional sense -- as families of elements in Galois cohomology -- associated to cohomological automorphic representations of $G(\AA)$. Combining this with known theorems relating automorphic representations of $G$ and of $\GL_3 / E$, we obtain the following:

  \begin{ltheorem}
   Let $\Pi$ be a RAECSDC\footnote{See \cref{def:RAECSDC}} automorphic representation of $\GL_3 / E$ which is unramified and ordinary at the primes $\fp \mid p$. Let $V_{\fP}(\Pi)$ be its associated Galois representation, and suppose this representation is irreducible. Then there exists a lattice $T_{\fP}(\Pi)^* \subset V_\fP(\Pi)^*$, and a collection of classes
   \[ \mathbf{c}_{\fm}^{\Pi} \in H^1_{\Iw}\left(E[\fm p^\infty], T_{\fP}(\Pi)^*\right) \]
   for all $\fm \in \cR$ coprime to $pc$, such that for all $\fm \mid \fn$ we have
   \[
    \operatorname{norm}_{\fm}^{\fn}\left(\mathbf{c}_{\fn}^{\Pi}\right) = \Big(\prod_{w \mid \frac{\fn}{\fm}} P_w(\Pi, \sigma_w^{-1})\Big) \mathbf{c}_{\fm}^{\Pi},
   \]
   where $P_w(\Pi, X) = \det(1 - X \operatorname{Frob}_w^{-1} : V_\fP(\Pi)(1))$.
  \end{ltheorem}

  See \cref{thm:ESpi} for a precise statement, and for some additional properties of the classes $\mathbf{c}_{\fm}^{\Pi}$. As well as constructing these Euler systems, we also prove interpolation results showing that their $p$-adic \'etale realisations are compatible with twisting by $p$-adic families of algebraic Gr\"ossencharacters, and with variation in Hida families of automorphic representations.

  In future work, we will prove an explicit reciprocity law for this Euler system, relating it to values of an appropriate $p$-adic $L$-function, and thus prove the Bloch--Kato conjecture in analytic rank 0 for automorphic Galois representations arising from $G$. However, in the present paper we shall focus solely on the construction of the Euler system classes.

 \subsection{Outline of the paper}

  After some preliminary material presented in \cref{sect:prelims}, Sections \ref{sect:formalism}--\ref{sect:zeta} of this paper are devoted to proving a certain purely local, representation-theoretic statement which we call an ``abstract norm relation'' (\cref{thm:absnorm2}). This states that, if $\mathfrak{Z}$ is \emph{any} map from a certain space of local test data to a representation of $G(\QQ_\ell)$, satisfying an appropriate equivariance property, then the values of $\mathfrak{Z}$ on two particular choices of the test data are related by a certain specific Hecke operator $\cP$. We prove this in two stages. Firstly, in \S\ref{sect:spherical}, we prove that such a Hecke operator $\cP$ must exist (without identifying the operator), using a cyclicity result for Hecke modules inspired by work of Sakellaridis. Secondly, in \S\ref{sect:localdata1} and \S\ref{sect:zeta} we use local zeta integrals to define a directly computable, purely local example of a morphism $\mathfrak{z}$ with the correct equivariance property, which allows us to identify the relevant Hecke operator $\cP$ explicitly. We have developed this theory in some detail, since we expect that the strategy developed here will be applicable to many other Euler system constructions, and it might also serve to clarify some possibly confusing details in earlier works of ours such as \cite{LSZ17}.

  In the second part of the paper, Sections \ref{sect:lietheory}--\ref{sect:UEconstruct}, we construct a second, much more sophisticated example of a morphism to which the above theory applies: the ``unitary Eisenstein map'' $\UE^{[a,b,r,s]}$ of \cref{def:uEis}, taking values in the motivic cohomology of the $\operatorname{GU}(2, 1)$ Shimura variety. Applying the ``abstract norm relation'' to this specific choice of morphism, we obtain a family of motivic classes satisfying norm-compatibility relations, whose denominators are uniformly bounded in the \'etale realisation. This is our Euler system.

  In the final sections of the paper, we prove that these classes satisfy norm-compatibility relations in a suitable tower of levels at $p$, and that their \'etale realisations are compatible with certain $p$-adic moment maps arising from this tower. This can be interpreted as stating that the \'etale  Euler-system classes vary analytically in Hida families for $G$; this is an important input for studying explicit reciprocity laws for the Euler system, which will be the subject of a forthcoming paper. Finally, we briefly discuss the Euler system for an individual automorphic Galois representation obtained by projecting our classes to a cuspidal Hecke eigenspace.

 \subsection*{Acknowledgements}

  We are grateful to Yiannis Sakellaridis for his explanations regarding the cyclicity result of \cite{sakellaridis13} and its generalisations. This work was begun while the first and third authors were visiting the Institute for Advanced Study in the spring of 2016, and we are very grateful to the IAS for their hospitality. We also thank the anonymous referee for their careful reading of the manuscript.

\section{The groups $G$ and $H$}
 \label{sect:prelims}

 \subsection{Fields}

  Let $E$ be an imaginary quadratic field, of discriminant $-D$, and let $x \mapsto \bar{x}$ be the nontrivial automorphism. Let $\cO$ be the ring of integers of $E$. We fix an identification of $E\otimes \RR$ with $\CC$ such that $\delta = \sqrt{-D}$ has positive imaginary part.

 \subsection{The group $G$}

  Let $J \in \GL_3(E)$ be the Hermitian matrix
  \[
   J = \begin{smatrix} & & \delta^{-1} \\ & 1 & \\ -\delta^{-1} & & \end{smatrix} \in \GL_3(E),\qquad \delta = \sqrt{-D}.
  \]

  \begin{definition}
   Let $G$ be the group scheme over $\ZZ$ such for that a $\ZZ$-algebra $R$
   \[
    G(R) = \left\{ (g,\nu) \in \GL_3(\cO\otimes R)\times R^\times:
    \t\bar{g} \cdot J \cdot g = \nu J\right\}.
   \]
   We identify $Z_G$ with $\Res_{\cO/\ZZ}(\GG_m)$, via $z \mapsto (\sdth{z}{z}{z}, z\bar{z})$. We write $\mu: G \to \Res_{\cO/\ZZ}(\GG_m)$ for the character $(g, \nu) \mapsto \tfrac{\det \bar{g}}{\nu}$, so $\mu \bar{\mu} = \nu$.
  \end{definition}

  The real group $G(\RR)$ is the unitary similitude group $\operatorname{GU}(2,1)$; see e.g.~\cite[\S 2.2]{pollackshah18}. Note that $G$ is reductive over $\ZZ_\ell$ for all $\ell \nmid D$ (even if $\ell = 2$).

  \begin{lemma}
   \label{lem:boreldecomp}
   Let $B_G\subset G$ be the upper-triangular subgroup. Then $B_G=T_G\ltimes N_G$, with
   \[
    T_G(R) =\left\{ \left(\sdth{x}{x}{x} \sdth{z\bar{z}}{\bar{z}}{1}, x\bar{x}z\bar{z}\right) : x,z \in (\cO\otimes R)^\times\right\}
   \]
   the diagonal torus and
   \[
    N_G(R) = \left\{\left(\begin{pmatrix} 1 & \delta s & t + \epsilon s\bar{s} \\ & 1 & \bar s \\ & & 1 \end{pmatrix},1\right) \ : \ s\in \cO\otimes R, t\in R\right\}.
   \]
   Here $\epsilon = \tfrac{1 + \delta}{2}$ if $D$ is odd, and $\epsilon = \tfrac{\delta}{2}$ otherwise. Given $s,t$ as above, we will write $t(x, z) \in T_G(R)$ and $n(s,t)\in N_G(R)$ for the corresponding elements. We abbreviate $t(1, z)$ as $t(z)$. Note that
   \[ t(z) \cdot n(s, t) \cdot t(z)^{-1} = n(zs, z\bar{z} t).\]

   We write $\bar{B}_G$ and $\bar{N}_G$ for the lower-triangular Borel and its unipotent radical.
  \end{lemma}

  \begin{lemma}
   \label{lem:isooverE}
   If $R$ is an $\cO[1/D]$-algebra, the map $i: \cO \otimes_{\ZZ} R \rightarrow R$ given by $x\otimes y\mapsto xy$ gives an isomorphism of group schemes
   \[ G \times_{\ZZ} R \cong (\GL_3\times \GG_m)/R, \qquad (g, \nu) \mapsto (i(g), \nu).\]
  \end{lemma}

 \subsection{The group $G_0$}

  We define $G_0 = \ker(\nu) \subset G$, so $G_0$ is the group of unitary isometries (as opposed to unitary similitudes) of $J$. Since $\frac{g}{\mu(g)} \in G_0$ for all $g \in G$, we have
  \begin{equation}
   \label{eq:G0}
   G_0(R) Z_G(R) = G(R)
  \end{equation}
  for all $\ZZ$-algebras $R$.

 \subsection{The group $H$}
  Let $H$ be the group scheme over $\ZZ$ such that for a $\ZZ$-algebra $R$
  \[
   H(R) = \{(g,z)\in\GL_2(R)\times (\cO\otimes R)^\times \ : \ \det(g) = z\bar z \}.
  \]
  This can be identified with a subgroup of $G$:
  \[
   \iota:H\into G, \ \ (\left(\smallmatrix a & b \\ c & d\endsmallmatrix\right),z) \mapsto
   (\left(\smallmatrix a &  & b \\  & z &  \\ c &  & d\endsmallmatrix\right),z\bar z).
  \]
  In particular we can regard $\mu$ as a character of $H$, by composition with $\iota$, and we have simply $\mu(\, (g, z)\, ) = \bar{z}$.

  \begin{note}\label{note:splitlH}
   If $\ell$ is a prime split in $E$, and we fix a prime $w \mid \ell$ of $E$ as above, then $w$ gives an embedding $\cO[1/D] \into \Zl$. So \cref{lem:isooverE} gives an identification $G(\Ql) \cong \GL_3(\Ql) \times \Ql^\times$. We also have an isomorphism $H(\Ql) \cong \GL_2(\Ql) \times \Ql^\times$, given by $(\gamma, z) \mapsto (\gamma, i(z))$. Via these identifications, $\iota: H \into G$ corresponds to the map $\GL_2 \times \GG_m \to \GL_3 \times \GG_m$ given by
   \[
    \left[\stbt a b c d,\, x\right]\mapsto
    \left[\begin{smatrix} a && b\\ &x&\\c && d\end{smatrix},\,ad-bc\right].\qedhere
   \]
  \end{note}

 \subsection{Open orbits}

  The following relationship between $G$ and $H$ is crucial for our arguments:

  \begin{lemma}
   \label{lem:openorbit}
   Let $R$ be a $\ZZ[1/D]$-algebra, and let $Q_H^0$ be the subgroup $\{ (g, z) \in H: g = \stbt{\star}{\star}{0}{1}\}$. Then there exists an element $u \in N_G(R)$ such that the map
   \[ Q_H^0 \times \bar{B}_G \to G, \qquad (h, \bar{b}) \mapsto h u \bar{b}\]
   is an open immersion of $R$-schemes.
  \end{lemma}

  \begin{proof}
   We shall show that $u = n(1, 0)$ has this property.

   Clearly $(h, \bar{b}) \mapsto h u \bar{b}$ is an open immersion if and only if the translated map $\psi: (h, \bar{b}) \mapsto u^{-1} h u \bar{b}$ is an open immersion. Since $Q_H^0$ is contained in $H \cap B_G$, this map $\psi$ factors through the ``big Bruhat cell'' $N_G \times T_G \times \bar{N}_G$, which is well-known to be open in $G$. So it suffices to show that $\psi$ is an open immersion into the big Bruhat cell, or, equivalently, that the composite
   \[ Q_H^0 \xrightarrow{h \mapsto u^{-1}hu} B_G \twoheadrightarrow B_G / T_G = N_G\]
   is an open immersion. After a mildly tedious matrix manipulation one sees that this map is given by
   \[
    (\stbt{z\bar{z}}{y}{0}{1}, z) \mapsto n\Big(z - 1, y + (\bar{z}-1)\varepsilon + (z-1)\bar{\varepsilon}\Big).
   \]
   This clearly identifies $Q_H^0$ with the open subscheme of $N_G$ consisting of the $n(s, t)$ with $s \ne -1$.
  \end{proof}

  \begin{remark}
   The openness of the image amounts to the claim that $\bar{B}_G \times Q^0_H$, or equivalently $B_G \times B_H$, has an open orbit on the homogenous $(G \times H)$-variety $\sX = H \bs (G \times H)$ (where $H$ is embedded diagonally in $G \times H$). In other words, $\sX$ is a \emph{spherical variety}. This fact will play a crucial role in the norm-compatibility relations for our Euler system, both in the ``tame direction'' (see \cref{cyc-thm}) and the ``$p$-direction'' (\cref{thm:wildnormint}).
  \end{remark}

 \subsection{Base change and $L$-factors}
  \label{sect:basechange}

  We now relate representations of $G$ with representations of the group $\operatorname{Res}_{E/\QQ}(\GL_3 \times \GL_1)$.

  \subsubsection*{Local case} For each prime $\ell$ split in $E /\QQ$, and each prime $w \mid \ell$ of $E$, the prime $w$ determines an isomorphism of $G(\Ql)$ with $\GL_3(\Ql) \times \Ql^\times$, as above.

  \begin{definition}
   If $\pi_\ell$ is an irreducible smooth representation of $G(\Ql)$, we let $\bc_w(\pi_\ell)$ denote the representation of $\GL_3(\Ql) \times \Ql^\times$ obtained from $\pi_\ell$ via this isomorphism.

   If $\tau_w \boxtimes \psi_w = \bc_w(\pi_\ell)$, then we write $\operatorname{BC}_w(\pi_\ell)$ for the representation $\tau_w \otimes (\psi_w \circ \det)$ of $\GL_3(\Ql)$, and $L_w(\pi_\ell, s)$ for the $L$-factor $L(\operatorname{BC}_w(\pi_\ell), s)$.
  \end{definition}

  If $v$ is a place which does not split (including the infinite place), and $w$ the place above $v$ in $E$, then there is also a base-change map $\bc_w$ taking \emph{tempered} representations of $G(\QQ_v)$ to tempered representations of $(\GL_3 \times \GL_1)(E_w)$; this is a consequence of the local Langlands correspondence for unitary groups due to Mok \cite[Theorem 2.5.1]{mok15}. (See \cite[Definition 3.5]{pollackshah18} for explicit formulae when $\ell \nmid D$ and $\pi_\ell$ is spherical.) As in the split case, if $\bc_w(\pi_v) = \tau_w \boxtimes \psi_w$, we use the notation $L_w(\pi_v, s)$ for $L(\tau_w \otimes \psi_w, s)$.

  In either case we write $L(\pi_{v}, s) = \prod_{w \mid v} L_w(\pi_v, s)$, which is the $L$-factor associated to $\pi_v$ and the natural 6-dimensional representation of the $L$-group of $G$.

 \subsubsection*{Global case} (The definitions in this section will not be used until \S\ref{sect:galcoh}.) We recall the following definition (see e.g.~\cite[\S 1]{BLGHT11}):

  \begin{definition}\label{def:RAECSDC}
   A ``RAECSDC'' (regular algebraic, essentially conjugate self-dual, cuspidal) automorphic representation of $\GL_3 / E$ is a pair $(\Pi, \omega)$, where $\Pi$ is a cuspidal automorphic representation of $\GL_3 / E$ and $\omega$ is a character of $\AA^\times / \QQ^\times$, such that:
   \begin{itemize}
    \item $\Pi_\infty$ is regular algebraic (or, equivalently, cohomological)
    \item $\Pi^c \cong \Pi^\vee \otimes (\omega \circ \operatorname{N}_{E/\QQ})$, where $\operatorname{N}_{E/\QQ}$ is the norm map, and $\Pi^c$ the composite of $\Pi$ and the involution $x \mapsto \bar{x}$ on $\GL_3(\AA_E)$.
   \end{itemize}
   We say $\Pi$ is RAECSDC if there exists some $\omega$ such that $(\Pi, \omega)$ is RAECSDC.
  \end{definition}

  \begin{theorem}[Mok]
   \label{thm:autdescent}
   Let $(\Pi, \omega)$ be a RAECSDC automorphic representation of $\GL_3 / E$. Then there exists a unique globally generic, cuspidal automorphic representation $\pi$ of $G$ such that $BC_w(\pi_v) = \Pi_w$ for every prime $w$ of $E$, where $v$ is the place of $\QQ$ below $w$, and $\pi$ has central character $\chi_{\pi}^c / (\omega \circ \operatorname{N}_{E/\QQ})$. Moreover, $\pi$ is essentially tempered for all places $v$, and $\pi_\infty$ is cohomological for $G(\RR)$; and $\pi$ has multiplicity one in the discrete spectrum of $G$.
  \end{theorem}

  \begin{proof}
   We briefly indicate how to deduce this from the results of \cite{mok15} (which are formulated for $G_0$ rather than $G$). Let $\psi$ be the character $\chi_{\Pi} / (\omega \circ \operatorname{N}_{E/\QQ})$. Then the representation $\tau = \Pi \otimes \psi^{-1}$ is regular algebraic and conjugate self-dual; so by Example 2.5.8 of \emph{op.cit.} it descends to a generic $L$-packet for $G_0$, all of whose members have multiplicity one in the discrete spectrum of $G_0$. In particular, this $L$-packet has a unique generic member $\pi_0$. From the compatibility with local base-change, one computes that the central character of $\pi_0$ has to be the restriction of $\psi^c$ to $Z_{G_0}$. Hence, by \eqref{eq:G0}, the representation $\pi_0$ extends uniquely to a representation $\pi$ of $G$ with central character $\psi^c$, whose base-change is $\tau \boxtimes \psi$; and $\pi$ has multiplicity one in the discrete spectrum of $G$ by the argument of \cite[\S 1.1]{clozelharrislabesse}.
  \end{proof}

  \begin{remark}
   Our definitions are chosen in such a way that twisting $\pi$ by $\alpha \circ \mu$, for $\alpha$ a character of $\AA_E^\times / E^\times$, corresponds to twisting $\Pi$ by $\alpha \circ \det$ (and replacing $\omega$ with $\omega \cdot \alpha|_{\AA_{\QQ}^\times}$). This is the motivation for the apparently rather arbitrary definition of the character $\mu$.
  \end{remark}

  \begin{definition}
   We say that a cohomological automorphic representation $\pi$ of $G(\AA)$ is \emph{non-endoscopic} if it arises from the above construction for some RAECSDC representation $(\Pi, \omega)$ (or, equivalently, if $\pi$ is globally generic and $\operatorname{BC}(\pi)$ is cuspidal).
  \end{definition}

  \begin{remark}
   Note that not all regular algebraic cuspidal representations of $G$ arise from this construction: there are other ``endoscopic'' representations, arising by functoriality from $U(1, 1) \times U(1)$ or $U(1)^3$, which are cuspidal but have non-cuspidal base-change to $\GL_3$. However, these representations are not interesting from the perspective of constructing Euler systems, since they correspond to globally reducible Galois representations.
  \end{remark}

%%%%%%%%%%%%%%%%%%%%%%%%%%%%%%%%%%%%%%%%%%%%%%%%%%%%%%%%%%%%%%%%%%%%%%%%%%%%%%%%
\section{Formalism of equivariant maps}
\label{sect:formalism}
%%%%%%%%%%%%%%%%%%%%%%%%%%%%%%%%%%%%%%%%%%%%%%%%%%%%%%%%%%%%%%%%%%%%%%%%%%%%%%%%

 \subsection{Definitions}

  Let $S$ be a nonempty set of (rational) primes and let $\QQ_S$ denote the restricted direct product of the $\QQ_\ell$ for $\ell \in S$. We let $G_S = G(\QQ_S)$ and similarly $H_S$.

  Let $L$ be any field of characteristic 0, and write $\cS(G_S, L)$ for the space\footnote{This is the ``Hecke algebra'' of $G_S$, but the algebra structure depends on a choice of Haar measure on $G$, and we shall avoid making a choice for the moment and thus not use the algebra structure yet.} of compactly-supported, locally-constant $L$-valued functions on $G_S$.  We write $\cS(\QQ_S^2, L)$ for the space of Schwartz functions on $\QQ_S^2$.

  \begin{definition}
   \label{def:equivariant}
   Let $\mathcal{V}$ be a smooth $L$-linear (left) representation of $G_S$. We shall say an $L$-linear map
   \[ \mathfrak{Z}: \cS_{(0)}\left(\QQ_S^2, L\right) \otimes_L \cS\left(G_S, L\right) \to \mathcal{V} \]
   is \emph{$G_S \times H_S$-equivariant} if it is equivariant for the following (left) actions of $G_S \times H_S$:
   \begin{itemize}
    \item $G_S$ acts on the left-hand side by $g \cdot ( \phi \otimes \xi ) = \phi \otimes \xi( (-)g)$, and on the right-hand side by its given action on $\mathcal{V}$;
    \item $H_S$ acts on the left-hand side by $h \cdot (\phi \otimes \xi) = \phi( (-)h) \otimes \xi(h^{-1}(-))$, and trivially on the right-hand side.
   \end{itemize}
   Equivalently, these are the $G_S$-equivariant maps $\cI(G_S, L) \to \mathcal{V}$, where $\cI(G_S, L)$ is the $H_S$-coinvariants of $\cS\left(\QQ_S^2, L\right) \otimes_L \cS\left(G_S, L\right)$.
  \end{definition}

  We can make similar definitions with $\cS$ replaced with the space $\cS_0(\QQ_S^2, L)$ of Schwartz functions vanishing at $(0, 0)$; we write $\cI_0(G_S, L)$ for the $H_S$-coinvariants of $\cS_0\left(\QQ_S^2, L\right) \otimes_L \cS\left(G_S, L\right)$. In order to avoid unnecessary repetition, we adopt the following notational shortcut:

  \begin{notation}
   We write $\cS_{(0)}\left(\QQ_S^2, L\right)$ to denote a statement which is valid for either $\cS$ or $\cS_0$, and correspondingly $\cI_{(0)}$.
  \end{notation}

  As in \cite[\S 3.9]{LSZ17}, once a Haar measure on $G_S$ is chosen, one can identify $\cI_{(0)}(G_S, L)$ with the compact induction $\operatorname{cInd}_{H_S}^{G_S}(\cS_{(0)}(\QQ_S^2, L))$. It then follows from Frobenius reciprocity that $G_S$-equivariant maps $\cI_{(0)}(G_S, L) \to \mathcal{V}$ biject with $H$-invariant bilinear forms $\cS_{(0)}\left(\QQ_S^2, L\right) \otimes \mathcal{V}^\vee \to L$, where $\mathcal{V}^\vee$ is the smooth dual of $\mathcal{V}$ as a $G_S$-representation. (However, this bijection is not entirely canonical, since it depends on a choice of Haar measure on $G_S$.)

  \begin{definition}
   Let $U$ be an open compact subgroup of $G_S$. We shall write $\cI_{(0)}(G_S / U, \QQ)$ for the image in $\cI_{(0)}(G_S, \QQ)$ of the $U$-invariants $\cS_{(0)}\left(\QQ_S^2, L\right) \otimes \cS(G_S / U, L)$.
  \end{definition}

 \subsection{Integrality}

  Let us fix a Haar measure $\vol_{H, S}$ on $H_S$, which we suppose to be $\QQ$-valued.

  \begin{definition}
   \label{def:integral}
   We shall say an element of $\cI_{(0)}(G_S / U, \QQ)$ is \emph{primitive integral at level $U$} if it can be written in the form $\phi \otimes \ch(g  U)$ for some $\phi \in \cS_{(0)}$ and $g \in G_S$, and the function $\phi$ takes values in the fractional ideal $C \ZZ$, where we define
   \[ C = \frac{1}{\vol_{H, S}\left(g U g^{-1} \cap \operatorname{stab}_{H_S}(\phi)\right)}.\]
   An element of $\cI_{(0)}(G_S / U, \QQ)$ is said to be \emph{integral at level $U$} if it is a sum of primitive integral elements at level $U$; and we write the set of such elements as $\cI_{(0)}(G_S / U, \ZZ)$.
  \end{definition}

  Clearly, any element of $\cI_{(0)}(G_S / U, \QQ)$ can be scaled into $\cI_{(0)}(G_S /U, \ZZ)$. More generally, we can replace $\QQ$ with a number field $L$, and $\ZZ$ with $\cO_L[1 /\Sigma]$ for any set of primes $\Sigma$ of $L$.

  \begin{remark}
   This definition may seem bizarre at first sight; its motivation is the following. Later in this paper, we shall construct $G_S \times H_S$-equivariant maps into the motivic and \'etale cohomology of Shimura varieties for $G$, analogous to the ``Lemma--Eisenstein map'' considered in \cite{LSZ17} for the $\operatorname{GSp}_4$ case. However, the definition of these maps involves various volume factors, so it is far from obvious \emph{a priori} which input data give rise to classes in the integral \'etale cohomology. The above notion of ``integral elements'' is designed for exactly this purpose.
  \end{remark}

  Note that the definition of integrality depends on the level $U$, but we have the following compatibilities. For any $U' \subseteq U$ open compacts, we have an inclusion $\cS(G / U, \QQ) \into \cS(G / U', \QQ)$, and a trace map $\cS(G / U', \QQ) \to \cS(G / U, \QQ)$ mapping $\xi$ to $\sum_{\gamma \in U / U'} \xi((-)\gamma)$. Tensoring with the identity of $\cS_{(0)}(\QQ_S^2)$ gives maps $\cI_{(0)}(G_S / U, \QQ) \into \cI_{(0)}(G_S / U', \QQ)$ (``pullback'') and $\cI_{(0)}(G /  U', \QQ) \to \cI_{(0)}(G /  U, \QQ)$ (``pushforward''), whose composite is multiplication by $[U : U']$ on $\cI_{(0)}(G_S / U, \QQ)$.

  \begin{proposition}
   The above maps restrict to maps $\cI_{(0)}(G_S / U, \ZZ) \into \cI_{(0)}(G_S / U', \ZZ)$ and $\cI_{(0)}(G /  U', \ZZ) \to \cI_{(0)}(G /  U, \ZZ)$ respectively.
  \end{proposition}

  \begin{proof}
   Evidently, it suffices to check either statement on primitive integral elements. For the trace map this is selfevident, as the trace sends a coset $\ch(gU')$ to $\ch(gU)$, and the corresponding normalising factors $C'$ and $C$ satisfy $C' \mid C$, so primitive integral elements map to primitive integral elements. The reverse-direction map is a little more intricate, and follows by considering the orbits of the group $V = g U g^{-1} \cap \operatorname{stab}_{H_S}(\phi)$ on the $U'$-cosets contained in a given $U$-coset.
  \end{proof}

  \begin{remark}
   One can interpret the system of abelian groups $\cI_{(0)}(G_S / U, \ZZ)$, for varying $U$, as a ``Cartesian cohomology functor'' in the sense of \cite{loeffler-spherical}.
  \end{remark}

%%%%%%%%%%%%%%%%%%%%%%%%%%%%%%%%%%%%%%%%%%%%%%%%%%%%%%%%%%%%%%%%%%%%%%%%%%%%%%%%
\section{Spherical Hecke algebras and cyclicity}
\label{sect:spherical}
%%%%%%%%%%%%%%%%%%%%%%%%%%%%%%%%%%%%%%%%%%%%%%%%%%%%%%%%%%%%%%%%%%%%%%%%%%%%%%%%
 \subsection{Where we are going}

  Let $\ell$ be an odd prime unramified in $E$, and set $G_\ell = G(\Ql)$ and $H_\ell$ similarly. We normalise the Haar measures by $\vol_{H_\ell}(H^0_\ell) = 1$, where $H^0_\ell = H(\Zl)$, and similarly for $G$. For $w \mid \ell$ a prime of $E$, we define
  \[ G^0_\ell[w] = \{ g \in G^0_\ell: \mu(g) = 1 \bmod w\}.\]

  We would \emph{like} to prove the following statement (an ``abstract norm relation''): if $\delta_0 = \ch(\Zl^2) \otimes \ch(G^0_\ell)$ is the natural spherical vector of $\cI(G_\ell / G^0_\ell, \ZZ)$, then there exists an element
  \[
   \delta_w \in \cI\left(G_\ell / G^0_\ell[w], \ZZ\right)\quad\text{such that}\quad
    \operatorname{norm}_{G^0_\ell}^{G^0_\ell[w]}\left(\delta_w\right) = \cP'_w(1) \cdot \delta_0,
  \]
  where $\cP'_w$ (to be defined below) is a certain polynomial over the spherical Hecke algebra, related to local Euler factors. What we shall actually prove, as \cref{thm:absnorm2} below, is something a little weaker than this, but still sufficient for applications: $\delta_w$ is only integral up to powers of $\ell$, and if $\ell$ is inert, the equality $\operatorname{norm}_{G^0_\ell}^{G^0_\ell[w]}\left(\delta_w\right) = \cP'_w(1) \cdot \delta_0$ only holds up to inverting a certain element in the centre of the Hecke algebra.

  We shall prove this statement in two stages. Firstly, we shall show that for any open $U \subseteq G^0_\ell$ and any $\delta \in \cI\left(G_\ell / U, \ZZ\right)$, there exists an element $\cP_\delta$ lying in (a localisation of) the spherical Hecke algebra of $G_\ell$ such that $\operatorname{norm}_{G^0_\ell}^{U}\left(\delta\right) = \cP_\delta \cdot \delta_0$. This relies crucially on a cyclicity result for Hecke algebras due to Sakellaridis (\cref{cyc-thm}).

  Secondly, we shall write down a candidate for $\delta_w$ and verify that it is integral at level $G^0_\ell[w]$ up to powers of $\ell$. The aforementioned results then show that $\operatorname{norm}_{G^0_\ell}^{G^0_\ell[w]}\left(\delta_w\right)$ is the image of $\delta_0$ under some Hecke operator $\cP_{\delta_w}$. Via a lengthy but routine computation with local zeta integrals, we show that this Hecke operator must be equal to $\cP_w'(1)$. This completes the proof.

 \subsection{Preliminaries}

  As in the previous section, let $\ell \nmid D$ be a prime. From here until the end of \cref{sect:spherical}, all Schwartz spaces and Hecke algebras are over $\CC$ and we omit this from the notation.

  \subsubsection{Hecke algebras}

   Let $\cH_{G,\ell}$ denote the Hecke algebra, whose underlying vector space is $\cS(G_\ell)$ and whose algebra structure is given by convolution with respect to some choice of Haar measure $\mathrm{d}x$:
   \[ (\xi_1 \star \xi_2)(x) = \int_{g \in G_\ell}\xi_1(g) \xi_2(g^{-1}x)\dd g
   = \int_{g \in G_\ell}\xi_1(x g^{-1}) \xi_2(g)\dd g.\]

   Any smooth left representation of $G_\ell$ can be regarded as a left $\cH_{G,\ell}$-module, via the action
   \[ \xi\star  v=\int_{G_\ell} \xi(g)\, (g\cdot v)\dd g. \]
   In particular, if $\xi = \ch(gK)$ for some subgroup $K$, and $g$ is $K$-invariant, then $\xi \star v = \vol(K) g\cdot v$. Similar constructions apply to right modules; and these constructions are compatible with the $(\cH_{G,\ell}, \cH_{G,\ell})$-bimodule structure of $\cH_{G,\ell}$ itself, if we define
   \[ g_1 \cdot \xi \cdot g_2 = \xi\left(g_1^{-1}(-) g_2^{-1}\right).\]

   The same constructions apply likewise with $H_\ell$ in place of $G_\ell$. Since a smooth $G_\ell$-representation is in particular a smooth $H_\ell$-representation by restriction, we can regard such representations as modules over either $\cH_{G,\ell}$ or $\cH_{H,\ell}$, and if necessary we write $\star_G$ or $\star_H$ to distinguish between the two convolution operations.

   If $\xi \in \cH_{G, \ell}$, we write $\xi'$ for its pullback via the involution $g \mapsto g^{-1}$ of $G_\ell$, and similarly for $\cH_{H, \ell}$.

  \subsubsection{Spherical Hecke algebras}

   Let $G^0_\ell=G(\Zl)$ and $H^0_\ell=H(\Zl)$. These are hyperspecial maximal compacts of $G_\ell$ and $H_\ell$, respectively. We suppose that the Haar measures on $G_\ell, H_\ell$ are chosen such that $G^0_\ell$ and $H^0_\ell$ have volume 1. The associated spherical Hecke algebras
   \[
   \cH_{G, \ell}^0 = C_c(G^0_\ell \bs G_\ell /G^0_\ell),\qquad \cH_{H, \ell}^0 = C_c\left(H^0_\ell \bs H_\ell/H^0_\ell\right).
   \]
   are commutative rings, and can be described (via the Satake isomorphism) as Weyl-group invariant polynomials in the Satake parameters.

  \subsubsection{Equivariant maps}

   We write $[-]$ for the quotient map from $\cS(\Ql^2) \otimes \cH_{G, \ell}$ to its $H_\ell$-coinvariants $\cI(G_\ell)$, with the actions as given in \cref{def:equivariant}. An easy unravelling of definitions shows that
   \[
    \left[\phi \otimes (\xi_1 \star_G \xi_2) \right] = \xi_2' \star_G [\phi \otimes \xi_1]
   \]
   for all $\phi \in \cS(\Ql^2)$, $\xi_1, \xi_2 \in \cH_{G, \ell}$, and
   \[
    \left[(\chi \star_H \phi) \otimes \xi \right] = \left[\phi \otimes (\chi' \star_H \xi )\right]
   \]
   for all $\phi \in \cS(\Ql^2)$, $\xi \in \cH_{G, \ell}$, $\chi \in \cH_{H, \ell}$.

  \subsubsection{Cyclicity}

   We can consider the space
   \[
    \cH = \cS(H^0_\ell \bs G_\ell /G^0_\ell),
   \]
   of smooth, compactly supported functions $G_\ell \rightarrow \CC$ that are left $H^0_\ell$-invariant and right $G^0_\ell$-invariant. This is evidently a $(\cH^0_{\cH, \ell}, \cH^0_{G, \ell})$-bimodule, via the convolution operations $\star_H$ and $\star_G$.

   \begin{theorem}
    \label{cyc-thm}
    $\cH$ is cyclic as an $(\cH^0_{\cH, \ell}, \cH^0_{G, \ell})$-bimodule, generated by the characteristic function $\xi_0 = \ch(G^0_\ell)$ of $G^0_\ell$. That is, every $\xi \in \cH$ can be written as a finite sum $\sum_i \alpha_i \star_H\beta_i$, for $\alpha_i \in \cH^0_{\cH, \ell}$ and $\beta_i \in \cH^0_{G, \ell}$.
   \end{theorem}

   If $\ell$ is split, this can be deduced from Corollary 8.0.4 of \cite{sakellaridis13}, applied to the group $\sG= G \times H$, acting by right-translation on the quotient $\sX = H \bs (G \times H)$, where $H$ embeds into $G \times H$ via $(\iota, \mathrm{id})$. It follows easily from \cref{lem:openorbit} that $\sX$ is \emph{spherical} as a $\sG$-variety, i.e.~the Borel subgroup $B_{\sG} = B_G \times B_H$ has an open orbit on $\sX$. Sakellaridis' result shows that for any split reductive group $\sG$ over $\Zl$ and spherical $\sG$-variety $\sX$ satisfying a certain list of conditions, the space of $\sG(\Zl)$-invariant Schwartz functions on $\sX(\Ql)$ is cyclic as a module over the unramified Hecke algebra of $\sG$, generated by the characteristic function of $\sX(\Zl)$; applying this to our $\sG$ and $\sX$ gives the theorem.

   However, since the hypotheses of Sakellaridis' general result are not entirely straightforward to verify in our setting, and Sakellaridis' argument does not cover the non-split case, we shall give a direct proof in an appendix; see \cref{cyc-thm-2}.

   \begin{remark}
    This theorem implies, in particular, that if $\pi_\ell$ and $\sigma_\ell$ are irreducible unramified representations of $G_\ell$ and $H_\ell$ respectively, then any element of $\Hom_{H_\ell}(\pi_\ell \otimes \sigma_\ell, \CC)$ is uniquely determined by its value on the spherical vectors, so the Hom-space has dimension $\le 1$. This relates our present approach to that of \cite{LSZ17}, where a ``multiplicity $\le 1$'' statement of this kind was taken as a starting-point for proving norm relations.
   \end{remark}

 \subsection{Hecke action on Schwartz functions}

  \begin{definition}
   Let us write $A$ for the torus $H \cap \iota^{-1}(Z_G)$, and $z_A: \GG_m \xrightarrow{\cong} A$ the map sending $x \mapsto (\stbt x {}{} x, x)$.
  \end{definition}

  The spherical Hecke algebra $\cH^0_{A, \ell}$, with respect to the (unique) maximal compact $A^0_\ell = A(\Zl) \cong \Zl^\times$, is isomorphic to $\CC[X, X^{-1}]$, where $X = \ch( z_A(\ell) A^0_\ell)$.

  \begin{definition}
   We let $\Delta_G$ and $\Delta_H$ be the maps $\cH^0_{A, \ell} \to \cH^0_{G, \ell}$ and $\cH^0_{A, \ell} \to \cH^0_{H, \ell}$ mapping $z_A(\ell^t) A(\Zl)$ to $z_A(\ell^t)G(\Zl)$ and $z_A(\ell^t)H(\Zl)$ respectively.
  \end{definition}

  These maps are both injective, and their images are central subalgebras of $\cH^0_{G, \ell}$ and $\cH^0_{H, \ell}$ respectively.

  \begin{lemma}
   Let $\phi_0 = \ch(\Zl^2)$. There exists a unique homomorphism
   \[ \zeta_H: \cH^0_{H, \ell} \to \cH^0_{A, \ell} \]
   such that
   \[ \xi \cdot \phi_0 = (\Delta_H \circ \zeta_H)(\xi) \cdot \phi_0 \]
   for all $\xi \in \cH_{H, \ell}^0$, where we let $H_\ell$ act on the space $\cS(\Ql^2)$ via the natural projection $H_\ell \to \GL_2(\Ql)$.
  \end{lemma}

  \begin{proof}
   We first define a map $\zeta: \cH^0_{\GL_2, \ell} \to \cH^0_{A, \ell}$. It is well known that $\cH^0_{\GL_2, \ell} \cong \CC[ T_\ell, S_\ell^{\pm 1}]$ where $T_\ell$ and $S_\ell$ are the double cosets of $\stbt{\ell}{0}{0}{1}$ and $\stbt \ell 0 0 \ell$. We define $\zeta$ by
   \begin{align*}
    \zeta(T_\ell) &= X + \ell, &
    \zeta(S_\ell) &= X,
   \end{align*} where $X = \ch( z_A(\ell) A^0_\ell)$ as above. Now we extend this map to $H_\ell$, by composing with the natural map $\cH^0_{H, \ell} \to \cH^0_{\GL_2, \ell}$ which sends a coset $\ch( H_\ell^0 (\gamma, z) H_\ell^0)$ to $\ch(\GL_2(\Zl) \gamma \GL_2(\Zl))$.
  \end{proof}

  \begin{proposition}
   \label{prop:referee1}
   Let $\mathfrak{s}(\Ql^2)$ denote the $H_\ell$-submodule of $\cS(\Ql^2)$ generated by the spherical vector $\phi_0$. If $\ell$ is split in $E$, then we have $\mathfrak{s}(\Ql^2) = \cS(\Ql^2)^{A^0_\ell}$. If $\ell$ is inert, then the quotient $\cS(\Ql^2)^{A^0_\ell} / \mathfrak{s}(\Ql^2)$ is annihilated by $\Delta_H(z_A(\ell) + \ell)$.
  \end{proposition}

  \begin{proof}
   We show first that $\cS(\Ql^2)^{A^0_\ell}$ is cyclic as a $\CC[\GL_2(\Ql)]$-module. This is surely well-known, but we give a sketch proof for completeness. It suffices to show that the $\CC[\GL_2(\Ql)]$-span of $\phi_0$ contains $\cS_0(\Ql^2)$. We can decompose $\Ql^2 - \{0, 0\}$ as a disjoint union of countably many $\GL_2(\Zl)$-invariant compact subsets $X_n$, where $X_n = \{ (x, y): \min(v_p(x), v_p(y)) = n\}$. Since $\stbt{1}{0}{0}{\ell}$ gives a (continuous) bijection between $X_n$ and $X_{n+1}$, we are reduced to showing that $\cS(X_0)^{A^0_\ell} = \cS(\mathbf{P}^1(\Zl))$ is contained in the $\GL_2(\Ql)$-span of $\phi_0$. However, for any $t \ge 1$ this span contains the vector
   \begin{equation}
    \label{def:phit}
    \phi_t \coloneqq \ch( p^t\Zp \times \Zp^\times) = \left(\stbt{p^{-t}}{0}{0}{1} - \stbt{p^{-t}}{0}{0}{p^{-1}}\right) \phi_0
   \end{equation}
   and these are the characteristic functions of a basis of neighbourhoods of $(0:1)$ in $\mathbf{P}^1(\Zl)$. As $\GL_2(\Zl)$ acts transitively on $\mathbf{P}^1(\Zl)$, the translates of the $\phi_t$ span $\cS(\mathbf{P}^1(\Zl))$.

   Since $H_\ell$ surjects onto $\GL_2(\Ql)$ for $\ell$ split, this shows that $\mathfrak{s}(\Ql^2) = \cS(\Ql^2)^{A^0_\ell}$ in this case. In the inert case, if we write $\GL_2(\Ql) = \GL_2(\Ql)^+ \bigsqcup \GL_2(\Ql)^-$ according to the parity of the valuation of $\det g$, then the image of $H_\ell$ is $\GL_2(\Ql)^+$. By the preceding paragraph, we can write any $\phi \in \cS(\Ql^2)^{A^0_\ell}$ in the form $\left(\xi^{+} + \xi^{-}\right) \star \phi_0$, where $\xi^{?}$ is supported on $\GL_2(\Ql)^?$; and since $\Delta_H(z_A(\ell)+ \ell) - T_\ell$ annihilates $\phi_0$, we have
   \[ \Delta_H(z_A(\ell) + \ell) \star \phi = \left(\xi^{+} \star \Delta_H(z_A(\ell) + \ell) + \xi^{-} \star T_\ell\right) \star \phi_0,\]
   and both $\xi^{+} \star \Delta_H(z_A(\ell) + \ell)$ and $\xi^{-} \star T_\ell$ are supported on $\GL_2(\Ql)^+$ and hence in the image of $\cH_{H, \ell}$.
  \end{proof}

  \begin{remark}
   This result is essentially best possible, since the quotient $\cS(\Ql^2)^{A^0_\ell} / (z_A(\ell) + \ell)$ is isomorphic to the induced representation $I(|\cdot|^{-1/2}, |\cdot|^{-1/2})$. This is irreducible as a $\GL_2(\Ql)$-representation, but splits into two direct summands as a representation of $\GL_2(\Ql)^+$, and the spherical vector is contained in one of the summands. So $\mathfrak{s}(\Ql^2)$ consists precisely of the vectors whose projection to the non-spherical summand of $I(|\cdot|^{-1/2}, |\cdot|^{-1/2})$ is 0.
  \end{remark}

  \begin{theorem}
   Let $[\delta_0] = [\phi_0 \otimes \xi_0] \in \cI(G_\ell / G^0_\ell)$. If $\ell$ is split, then we have $\cI(G_\ell / G_\ell^0) = \cH^0_{G, \ell} \star [\delta_0]$. If $\ell$ is inert, the quotient $\cI(G_\ell / G_\ell^0) / \left( \cH^0_{G, \ell} \star [\delta_0]\right)$ is annihilated by $\Delta_G(z_A(\ell) + \ell)$.
  \end{theorem}

  \begin{proof}
   Let $\delta = \phi \otimes \xi$ be a general element of $\cI(G_\ell / G^0_\ell)$. If $\ell$ is split, then \cref{prop:referee1} shows that we can find some $\theta \in \cH(H_\ell / H^0_\ell)$ such that $\phi = \theta \star_H \phi_0$. Hence in $\cI(G_\ell / G_\ell^0)$ we have
   \[ \left[ \phi \otimes \xi\right]=[ ( \theta \star_H\phi_0 ) \otimes \xi] =
   \left[ \phi_0 \otimes ( \theta' \star_H \xi)\right].\]

   Let $\sigma = \theta' \star_H \xi$. Since $\theta$ is invariant under right-translation by $H^0_\ell$, and $\xi$ under right-translation by $G^0_\ell$, we conclude that $\sigma \in \cH$. By \cref{cyc-thm}, we can express $\sigma$ (possibly non-uniquely) as a finite sum $\sum_i \alpha_i \star_{H} \beta_i$ for $\alpha_i \in \cH^0_{H, \ell}$ and $\beta_i \in \cH^0_{G, \ell}$.

   We can then write
   \begin{align*}
    \left[ \phi_0 \otimes ( \theta' \star_H \xi)\right] &= \sum_i \left[ \phi_0 \otimes(\alpha_i \star_H \beta_i)\right]\\
    &=  \sum_i \left[ (\alpha'_i \star_H \phi_0) \otimes \beta_i]\right]\\
    &= \sum_i \left[ (\Delta_H(\zeta_i) \star_H \phi_0) \otimes \beta_i\right] \\
    &= \sum_i \left[ \phi_0 \otimes (\Delta_H(\zeta_i)' \star_H \beta_i)\right]\\
    &= \sum_i \left[ \phi_0 \otimes (\Delta_G(\zeta_i)' \star_G \beta_i)\right],
   \end{align*}
   where we write $\zeta_i = \zeta_H(\alpha'_i) \in \cH^0_{A, \ell}$. (The last equality follows since the actions of $\cH^0_{A, \ell}$ on $\cH^0_{G, \ell}$ via $\Delta_G$ and $\Delta_H$ are the same: both are just the natural translation action of $A_\ell$ on $G_\ell$.)

   So, if we set $\Lambda = \sum_i \Delta_G(\zeta_i)' \star_G \beta_i \in \cH^0_{G, \ell}$, then we have
   \[ [\phi \otimes \xi] = [ \phi_0 \otimes \Lambda] = \Lambda' \star_G [ \phi_0 \otimes \xi_0].\]
   If $\ell$ is inert, then we can still find $\theta$ such that $\theta \star_H \phi_0 = \Delta_H(z_A(\ell) + \ell) \phi_0$, and the same argument as above produces a $\Lambda$ such that
   \[ \Delta_G(z_A(\ell) + \ell) \star_G [\phi \otimes \xi] = \Lambda' \star_G [\phi_0 \otimes \xi_0],\]
   showing that $\Delta_G(z_A(\ell) + \ell)$ annihilates the class of $\phi \otimes \xi$ in $\cI(G_\ell / G_\ell^0) / \left( \cH^0_{G, \ell} \star [\delta_0]\right)$.
  \end{proof}

  \begin{corollary}[Abstract norm relation, version 1]
   \label{cor:absnorm1}
   Let $U \subseteq G^0_\ell$ be an open subgroup, and $\delta \in \cI(G / U)$. If $\ell$ is split, there exists an element $\cP_\delta \in \cH^0_{G, \ell}$ with the following property:
   \begin{quotation}
    For any smooth $G_\ell$-representation $\mathcal{V}$ and $G_\ell \times H_\ell$-equivariant map $\mathfrak{Z}: \cS\left(\Ql^2\right) \otimes \cH\left(G_\ell\right) \to \mathcal{V}$, we have
    \[
     \cP_\delta \star_G \mathfrak{Z}(\delta_0) =
     \operatorname{norm}^U_{G^0_\ell}\left(\mathfrak{Z}(\delta)\right).
    \]
   \end{quotation}
   If $\ell$ is inert, then we can find an element $\cP_\delta \in \cH^0_{G, \ell}\left[\frac{1}{\Delta_G(z_A(\ell) + \ell)}\right]$ having the same property for every $\mathcal{V}$ such that $\Delta_G(z_A(\ell) + \ell)$ is invertible on $\mathcal{V}^{G^0_\ell}$.
  \end{corollary}

  \begin{proof}
   Replacing $\delta$ with the sum of its translates by $U / G^0_\ell$, we may assume $U = G^0_\ell$, and the result is now obvious from the preceding theorem.
  \end{proof}

 \subsection{Characterising $\cP_\delta$}
  \label{sect:cPdelta}

  Let $\pi_\ell$ be an irreducible spherical representation of $G_\ell$. Then the Hecke algebra acts on the 1-dimensional space $(\pi_\ell)^{G_\ell^0}$ via a ring homomorphism $\Theta_{\pi_\ell}: \cH^0_{G, \ell} \to \CC$

  If $\ell$ is inert in $E$, we suppose that the central character $\chi_{\pi_\ell}$ satisfies $\chi_{\pi_\ell}(z_A(\ell)) \ne -\ell^{-1}$, so that $\Delta_G(z_A(\ell) +\ell)'$ acts invertibly on $\pi_\ell$; hence $\Theta_{\pi_\ell}$ extends to $\cH^0_{G, \ell}\left[\frac{1}{\Delta_G(z_A(\ell) + \ell)}\right]$.

  \begin{proposition}
   \label{prop:frobrecip}
   Let $\mathfrak{z} \in \Hom_{H_\ell}(\cS(\Ql^2) \otimes \pi_\ell, \CC)$; and let $U$, $\delta$, and $\cP_\delta$ be as in \cref{cor:absnorm1}. Write $\delta = \sum_i \phi_i \otimes \ch(g_i U)$; and let $\varphi_0$ be a spherical vector of $\pi_\ell$. Then we have
   \[ \sum_i \mathfrak{z}(\phi_i \otimes g_i \varphi_0) = \Theta_{\pi_\ell}(\cP_\delta') \cdot \mathfrak{z}(\phi_0 \otimes \varphi_0).
   \]
  \end{proposition}

  \begin{proof}
   As usual, we may assume $U = G^0_\ell$. The homomorphism $\mathfrak{z}$ determines a linear map $\cZ: \cS(\Ql^2) \otimes \cH(G_\ell) \to \CC$ sending $\phi \otimes \xi$ to $\mathfrak{z}(\phi, \xi \star_G \varphi_0)$. This map clearly factors through $\cI(G_\ell / G^0_\ell)$, and it is $\cH^0_{G, \ell}$-equivariant if we let $\xi \in \cH^0_{G, \ell}$ act on $\CC$ by $\Theta_{\pi_\ell}(\xi')$.

   If $\ell$ is split, then we have $[\delta] = P_\delta \star_G [\phi_0 \otimes \xi_0]$ as elements of $\cI(G_\ell / G^0_\ell)$; so we must have $\cZ(\delta) = \Theta_{\pi_\ell}(\cP_\delta') \cZ(\delta_0)$, which is exactly the formula claimed in the proposition. If $\ell$ is inert, then we replace $\cI(G_\ell / G^0_\ell)$ with its localisation $\cI(G_\ell / G^0_\ell)\left[ 1/(z_A(\ell) + \ell)\right]$.
  \end{proof}

\section{Choice of the data}
 \label{sect:localdata1}
 Let $\ell \nmid D$ be prime, and $w$ a prime of $E$ above $\ell$. Let $q \coloneqq \Nm(w) = \ell$ or $\ell^2$.

 \subsection{The operator $\cP_w$}

  If $\pi_\ell$ is an irreducible unramified representation of $G_\ell$, we write $\Theta_{\pi_\ell}$ for the associated character of the Hecke algebra $\cH^0_{G, \ell}$, as in \S \ref{sect:cPdelta} above.

  \begin{lemma}
   There is a cubic polynomial $\cP_w \in \cH^0_{G, \ell}[X]$ such that for any irreducible unramified representation $\pi_\ell$ of $G_\ell$, we have $\Theta_{\pi_\ell}(\cP_{w})(q^{-s}) = L_w(\pi_\ell, s)^{-1}$.
  \end{lemma}

  \begin{proof}
   This is immediate from the Satake isomorphism, since the coefficients of the $L$-factor are Weyl-group-invariant polynomials in the Satake parameters.
  \end{proof}

  \begin{remark}
   One can check that $\cP_w(X)$ has the form $1 - \frac{1}{q} \ch\left(G^0_\ell\, t(\varpi_w)\, G^0_\ell\right) X \mathop{+} $ higher order terms, where $\varpi_w$ is a uniformizer at $w$; however, for our arguments it is actually not necessary to write down $\cP_w$ explicitly.
  \end{remark}

 \subsection{The element $\delta_w$}

  \begin{definition}
   \label{def:phi1t}
   For $t\geq 1$, define $\phi_{1,t}\in\cS(\Ql^2, \ZZ)$ as the function
   \[ \phi_{1,t}=\ch\left(\ell^t\Zl \times ( 1+\ell^t\Zl)\right). \]
  \end{definition}

  Note that $\phi_{1, t}$ is fixed by the action of the group
  \[
   K_{H_\ell, 1}(\ell^t) \coloneqq \{ (\gamma, z) \in H(\Zl): \gamma \equiv \stbt{z\bar{z}}{\star}{0}{1} \bmod \ell^t \Zl\}.
  \]

  \begin{definition}
   \label{def:tamedata}
   We define an element $\xi_w \in \cH(G_\ell / G^0_\ell[w])$, and an integer $n_w$, as follows:
   \begin{enumerate}[(i)]
    \item Suppose $\ell = w \bar{w}$ is split in $E$. Then we take $\xi_w = \ch(G^0_\ell[w]) - \ch(n(a, 0) G^0_\ell[w])$, where $a \in E\otimes \Ql$ has valuation $-1$ at $w$ and $\ge 1$ at $\bar{w}$; and we set
    \[ n_w = \ell(\ell + 1)(\ell - 1)^2.\]

    \item For $\ell$ inert in $E$, we take $\xi = \ch(G^0_\ell[w]) - \ch(n(a, 0) G^0_\ell[w])$ where $a \in E \otimes \Ql$ has valuation $-1$; and we take
    \[ n_w = (\ell^2 - 1)^2.\]

   \end{enumerate}
   With these notations, in both cases we define
   \[ \delta_w \coloneqq n_w \cdot \phi_{1, 2} \otimes \xi_w \in \cI(G_\ell / G^0_\ell[w], \QQ).\]
  \end{definition}

  \begin{proposition}
   We have $\delta_w \in \cI\left(G_\ell / G^0_\ell[w], \ZZ[1/\ell]\right)$.
  \end{proposition}

  \begin{proof}
   A tedious explicit computation shows that the subgroup $V = \operatorname{stab}_{H_\ell}(\phi_{1, 2}) \cap \operatorname{stab}_{G_\ell}(\xi_w)$ is given by $\{ h \in K_{H_\ell, 1}(\ell^2) : \mu(h) = 1 \bmod w\}$ if $\ell$ is split, and $ \{ h \in K_{H_\ell, 1}(\ell^2) : \mu(h) = 1 \bmod \ell \Zl + \ell^2 \cO_{E, \ell}\}$ if $\ell$ is inert. So $[H(\Zl) : V] = \ell^2 (\ell - 1)^2(\ell+1) = \ell n_w$ in the former case, and $\ell^3(\ell^2-1)^2 = \ell^3 n_w$ in the latter case. Thus $n_w \in \frac{C}{\ell} \ZZ$, resp.~$\frac{C}{\ell^3} \ZZ$, where $C = \tfrac{1}{\vol(V)} = [H(\Zl): V]$ is as in \cref{def:integral}.
  \end{proof}

  \begin{theorem}[Abstract norm relation, version 2]
   \label{thm:absnorm2}
   Let $\delta_w \in \cI(G_\ell/G^0_\ell[w], \ZZ[1/\ell])$ be the element defined in \cref{def:tamedata}. Let $\mathcal{V}$ be a smooth  $G_\ell$-representation and $\mathfrak{Z}: \cS(\Ql^2) \otimes \cH_{G, \ell} \to \mathcal{V}$ a $H_\ell \times G_\ell$-invariant homomorphism. If $\ell$ is inert, suppose also that $\Delta_G(z_A(\ell) + \ell)$ acts bijectively on $\mathcal{V}^{G^0_\ell}$. Then we have
   \[ \operatorname{norm}_{G^0_\ell}^{G^0_\ell[w]}\left(\mathfrak{Z}(\delta_w)\right)= \cP'_w(1) \star \mathfrak{Z}(\delta_0).\]
  \end{theorem}

  \emph{Outline of proof}. We need to show that if $\delta = \delta_w$, then the operator $\cP_\delta$ of \cref{cor:absnorm1} is $\cP'_w(1)$. We will do this using \cref{prop:frobrecip} to compare the images of $\cP_w(1)$ and $\cP'_\delta$ under $\Theta_{\pi_\ell}$, for a sufficiently dense set of unramified representations $\pi_\ell$. More precisely, for all unramified representations $\pi_\ell$ which are \emph{generic} (admit a Whittaker model), we shall construct below a non-zero, $H(\Ql)$-equivariant bilinear form $\mathfrak{z} \in \Hom_{H_\ell} \left(\pi_\ell \otimes \cS(\Ql^2), \CC\right)$ using zeta integrals, and show that for this $\mathfrak{z}$ we have
  \begin{equation}
   \label{eq:z-eval}
   n_w \mathfrak{z}\Big(\phi_{1, 2} \otimes (1 - n(a,0))\varphi_0\Big) = L_w(\pi_\ell, 0)^{-1}\mathfrak{z}(\phi_0 \otimes \varphi_0)\quad \text{and} \quad \mathfrak{z}(\phi_0 \otimes \varphi_0) \ne 0.
  \end{equation}
  The left-hand side of this equality is $\cZ(\delta_w)$ in the notation of \cref{prop:frobrecip}, so we must have $\Theta_{\pi_\ell}(\cP'_{\delta_w}) = L_w(\pi_\ell, 0)^{-1}$. Thus $\cP'_{\delta_w} = \cP_w(1)$ modulo the kernel of $\Theta_{\pi_\ell}$. Since the characters $\Theta_{\pi_\ell}$ for which this construction applies are dense in the spectrum of the Hecke algebra, we must in fact have $\cP'_{\delta_w} = \cP_w(1)$ as required. It remains only to construct the homomorphism $\mathfrak{z}$ and prove \cref{eq:z-eval}; this will be carried out in the next section.

%%%%%%%%%%%%%%%%%%%%%%%%%%%%%%%%%%%%%%%%%%%%%%%%%%%%%%%%%%%%%%%%%%%%%%%%%%%%%%%%
\section{Zeta-integral computations}
\label{sect:zeta}
%%%%%%%%%%%%%%%%%%%%%%%%%%%%%%%%%%%%%%%%%%%%%%%%%%%%%%%%%%%%%%%%%%%%%%%%%%%%%%%%
 \subsection{The zeta integral}

  Let $\ell$ be a rational prime (for now we do not need to assume $\ell \nmid D$). If $e$ is an additive character $E \otimes \Ql \to \CC^\times$, we can extend it to a character of $N(\Ql)$ via $n(s, t) \mapsto e(s)$. We fix a choice of $e$ whose restriction to $E_w$ is non-trivial for all $w \mid \ell$, and denote the resulting character of $N(\Ql)$ by $e_N$.

  \begin{definition}
   An irreducible representation $\pi_\ell$ of $G_\ell$ is said to be \emph{generic} if it is isomorphic to a space of functions on $G_\ell$ transforming by $e_N$ under left-translation by $N(\Ql)$. If such a subspace exists, it is unique, and we call it the \emph{Whittaker model} $\cW(\pi_\ell)$.
  \end{definition}

  \begin{definition}
   Let $\pi_{\ell}$ be a generic representation of $G_\ell$. For every $W \in \cW(\pi_\ell)$, and $s \in \CC$, define
   \[
    Z(W,s) \coloneqq \int_{(E \otimes \Ql)^\times} W\left( t(z) \right)  |\Nm(z)|^{s-1}\ddt z,
   \]
   where $t(z) = (\diag(z\bar{z}, \bar{z}, 1), z\bar{z})$ as above.
  \end{definition}

  \begin{proposition}\label{prop:Zpropertiesgen} \
   \begin{enumerate}
    \item The integral converges for $\Re(s)\gg 0$, and has analytic continuation as a rational function of $q^{s}$.
    \item The functions $Z(W,s)$ for varying $W$ form a non-zero fractional ideal of $\CC[q^s, q^{-s}]$ containing the constant functions.
    \item Let $h \in B_H(\Ql)$, and write $h = (\stbt{a}{b}{}{d}, z)$. Then we have
    \[ Z\left(\iota(h) W, s\right)=
    \chi(d) |\tfrac{d}{a}|^{s-1} Z\left(W, s\right),\]
    where $\chi = \chi_{\pi_\ell}|_{\Qlt}$. In particular this is independent of $z$.
   \end{enumerate}
  \end{proposition}
  \begin{proof}
   Parts (1) and (2) are standard facts. Part (3) is a simple explicit computation.
  \end{proof}

  \begin{definition}[Godement--Siegel sections]
   Let $\phi\in\cS(\Ql^2, \CC)$. We write $f^{\phi}(-, \chi, s)$ for the function $\GL_2(\Ql) \to \CC(\ell^s, \ell^{-s})$ defined by
   \[
    f^{\phi}(g, \chi, s) = |\det g|^s \int_{\Ql^\times} \phi( (0, a)g) \chi(a) |a|^{2s}\ddt a.
   \]
  \end{definition}

  This is a meromorphic section of the family of principal-series representations $I_{\GL_2}\left(|\cdot|^{s-\frac{1}{2}}, \chi^{-1}|\cdot|^{\frac{1}{2}-s}\right)$, regular away from the poles of $L(\chi, 2s)$. See also \cite[\S 8.1]{LPSZ1}.

  \begin{definition}
   For $\phi \in \cS(\Ql^2, \CC)$, we define
   \[ \mathfrak{z}(W, \phi, s) = \int_{(B_H \bs H)(\Ql)} Z(\iota(h) W, s) f^{\phi}(h, \chi, s)\dd h\in \CC(q^{s}, q^{-s}). \]
   where the integral is well-defined by (3) above.
  \end{definition}

  \begin{remark}
   The zeta-integral $\mathfrak{z}(\dots)$ is denoted $I_\ell(\dots)$ in \cite[\S 3.3]{pollackshah18} (taking the characters $(\nu_1, \nu_2)$ \emph{loc.cit.} to be $(1, \chi^{-1})$). It is a variant of the zeta-integral for $\mathrm{U}(2, 1)$ considered in \cite[\S 3.6]{Gel-PS-unitary}.

   We expect that for any generic $\pi_\ell$, the ``common denominator'' of the $\mathfrak{z}(W, \phi, s)$ should coincide with the $L$-factor $L(\pi_\ell, s)$ defined using the local base-change lifting as in \S \ref{sect:basechange}. However, in the present work we only need this when $\ell$ and $\pi_\ell$ are unramified. Some ramified cases are established in \cite[\S 3.6 \& \S 8.3]{pollackshah18}.
  \end{remark}

 \subsection{Explicit formulae in the unramified case}

  We suppose henceforth that $\ell \nmid 2D$, that $\pi_\ell$ is an irreducible unramified principal series, and that the additive character $e$ has conductor 1. Then $\pi_\ell$ is generic, and its Whittaker model $\cW(\pi_\ell)$ has a unique spherical vector $W_{\pi_\ell, 0}$ such that $W_{\pi_\ell, 0}(1) = 1$.

  \begin{proposition}\label{thm:unrameval}
   We have $Z(W_{\pi_\ell, 0}, s) = \frac{L(\pi_\ell, s)}{L(\chi, 2s)}$, where $\chi = \chi_{\pi_\ell}|_{\Ql^\times}$ as above, and $L(\pi_\ell, s)$ is as in \cref{sect:basechange}.
  \end{proposition}

  \begin{proof}
   The values of $W_{\pi_\ell, 0}$ along the torus $T$ are given by an explicit formula in terms of the Satake parameters; see \cite{shintani76} for $\ell$ split, and \cite[\S 4.7]{Gel-PS-unitary} for $\ell$ inert. The result follows from these formulae by an explicit computation.
  \end{proof}

  \begin{corollary}\label{cor:propertiesZ}
   If $\phi_0 = \ch(\Zl^2)$, then we have $\mathfrak{z}(W_{\pi_\ell, 0}, \phi_0,s) = L(\pi_\ell, s)$.
  \end{corollary}

  \begin{proof}
   We note that $f^{\phi_0}(-, \chi, s)$ is a spherical vector with $f^{\phi_0}(1, \chi, s) = L(\chi, 2s)$, and $H(\Zl)$ surjects onto $(B_H \bs H)(\Ql)$.
  \end{proof}

 \subsection{Invariant bilinear forms}

  \begin{theorem}[{\cite[Theorem 7.11]{pollackshah18}}]
   \label{thm:zetaint}
   The limit
   \[ \mathfrak{z}(W, \phi) \coloneqq \lim_{s \to 0} \frac{\mathfrak{z}(W, \phi, s)}{L(\pi_\ell, s)} \]
   exists for all $W \in \cW(\pi_\ell)$ and $\phi \in \cS(\Ql^2)$, and defines a non-zero element of the space $\Hom_{H_\ell}( \cS(\Ql^2) \otimes \pi_\ell,  \CC)$ satisfying $\mathfrak{z}(W_{\pi_\ell, 0}, \phi_0) = 1$.
  \end{theorem}

  \begin{remark}
   Note that this is much stronger than we need for the proof of Theorem \ref{thm:absnorm2}; it would suffice to know that there is \emph{some} non-zero rational function $P(s)$ such that $\lim_{s \to 0}\frac{\mathfrak{z}(W, \phi, s)}{P(s)}$ is well-defined and not identically 0.
  \end{remark}

 \subsection{Unipotent twists} We want to evaluate the above integrals on certain ramified test data (still assuming $\pi_\ell$ itself to be unramified).

  \begin{definition}\label{def:etadef}
   Let $w$ be a prime above $\ell$, and let $a \in E\otimes \Ql$ be such that $v_w(a) = -1$, with $v_{\bar{w}}(a) \ge 1$ if $\ell$ is split. We define
   \[ \eta_w^{(a)} = n(a, 0) \in N(\Ql). \]
  \end{definition}

  \begin{proposition}
   \label{prop:Zvalue}
   The value $Z(\eta_w^{(a)} W_{\pi_\ell, 0}, s)$ is independent of the choice of $a$, and is given by
   \[ Z((1 - \eta_w^{(a)}) W_{\pi_\ell, 0}, s) = \tfrac{q}{q - 1}L_{\bar{w}}(\pi_\ell, s). \]
  \end{proposition}

  \begin{proof}
   In the split case, $Z((1-\eta_S^{(a)}) W_{\pi_\ell, 0}, s)$ is given by
   \[ \int_{(E\otimes\Ql)^\times} (1-e(az)) W_0(t(z)) |\Nm(z)|^{s-1}\ddt z =
   \sum_{m,n \ge 0}\left( \int_{w^m \bar{w}^n \cO^\times}(1- e(az))\ddt z \right) W_0(t(\varpi_w^m \varpi_{\bar{w}}^n)) \ell^{-(m+n)(s-1)} .\]
   The bracketed integral is zero if $m \ge 1$; if $m = 0$ it is $\frac{\ell}{\ell-1}$. Since we have
   \[ \sum_{n \ge 0} W_0(t(\varpi_{\bar{w}}^n)) \ell^{-n(s-1)} = L_{\bar{w}}(\pi_\ell, s),\]
   the result follows. The argument in the inert case is similar, using the fact that $\int_{\ell^n \cO^\times}(1- e(az))\ddt z$ is 0 if $n > 1$ and $\tfrac{\ell^2}{\ell^2 - 1}$ if $n = 0$.
  \end{proof}

  \begin{remark}
   \label{rmk:wwbarzeta}
   By the same methods, one can show that for a split prime $\ell = w \bar{w}$ we have
   \[ Z((1 - \eta_w^{(a)})(1 - \eta_{\bar{w}}^{(\bar{a})}) W_{\pi_\ell, 0}, s) = \tfrac{\ell^2}{(\ell- 1)^2}.\qedhere
   \]
  \end{remark}

  \begin{corollary}
   In the situation of \cref{prop:Zvalue}, we have
   \[ \mathfrak{z}\left((1 - \eta_w^{(a)}) W_{\pi_\ell, 0}, \phi_{1, 2}\right)  = \tfrac{1}{n_w} \cdot L_w\left(\pi_\ell,0\right)^{-1},\]%\cdot\mathfrak{z}(W_{\pi_\ell, 0}, \phi_0), \]
   where $n_w$ is as in \cref{def:tamedata}.
  \end{corollary}

  \begin{proof}
   As in \cite[\S 3.10]{LSZ17}, for any $W \in \cW(\pi_\ell)$, the values $\ell^{2t-2}(\ell^2 - 1)\cdot  \mathfrak{z}(W, \phi_{1, t}, s)$ are independent of $t$ for $t \gg 0$, and the limiting value is simply $Z(W, s)$.

   In our case, it suffices to take $t = 2$ since both $\ell\eta_w^{(a)}$ and its inverse have matrix entries in $\cO \otimes \Zl$, so the principal congruence subgroup modulo $\ell^2$ fixes $(1 - \eta_w^{(a)}) W_{\pi_\ell, 0}$. Since $n_w = \frac{q-1}{q} \cdot \ell^{2}(\ell^2 - 1)$, the computation of the limiting value is immediate from \cref{prop:Zvalue}.
  \end{proof}

  This completes the proof of \eqref{eq:z-eval}, and hence of \cref{thm:absnorm2}.\qed

%%%%%%%%%%%%%%%%%%%%%%%%%%%%%%%%%%%%%%%%%%%%%%%%%%%%%%%%%%%%%%%%%%%%%%%%%%%%%%%%
\section{Algebraic representations and Lie theory}
\label{sect:lietheory}
%%%%%%%%%%%%%%%%%%%%%%%%%%%%%%%%%%%%%%%%%%%%%%%%%%%%%%%%%%%%%%%%%%%%%%%%%%%%%%%%

 \subsection{Representations of $G$ and $H$}

  Since $G$ and $H$ are split over $E$, their irreducible representations over $E$ are parametrised by highest-weight theory.

  \begin{definition}
   We write $\chi_i$, $i = 1\dots 4$, for the four characters of $T_{/E}$ mapping $\diag\left(x,z,\tfrac{z\bar{z}}{\bar x}\right)$ respectively to $x, \bar x, \tfrac{xz}{\bar x}, \tfrac{ \bar x \bar{z}}{x}$.
  \end{definition}

  \begin{note}
   The characters $\chi_1$ and $\chi_2$ are the highest weights (with respect to $B_G$) of the natural 3-dimensional representation $V$ of $G$ and its conjugate $\bar{V}$. The characters $\chi_3$ and $\chi_4$ factor through the abelianisation of $G$: we have $\chi_3 = \frac{\det}{\nu} = \bar{\mu}$ and $\chi_4 = \mu$, where $\mu = \overline{\det}/\nu$ as above. Moreover, $\chi_3\chi_4=\nu$.
  \end{note}

  \begin{definition} \
   \begin{enumerate}
    \item For $a_1,a_2\geq 0$, denote by $V^{a_1, a_2}$ the representation of $G$ of highest weight $a_1 \chi_1 + a_2 \chi_2$.
    \item For  $b_1\geq 0$, let $W^{b_1}$ denote the representation $\Sym^b(\mathrm{std})$ of $H$, where $
    \mathrm{std}$ denotes the pullback to $H$ of the defining representation of $\GL_2$.

    \item If $V$ is any representation of $G$ or $H$, we write $V\{a_3, a_4\}$ for its twist by $\chi_3^{a_3} \chi_4^{a_4}$.
   \end{enumerate}
  \end{definition}

  Thus every irreducible representation of $G$ has the form $V^{a_1, a_2}\{a_3, a_4\}$ for some $a_1, \dots, a_4 \in \ZZ$ with $a_1, a_2 \ge 0$; and every irreducible representation of $H$ has the form $W^{b_1}\{b_2, b_3\}$ for $b_1, \dots, b_3 \in \ZZ$ with $b_1 \ge 0$.

  \begin{note}
   We have
   \[ (V^{a_1, a_2})^*\cong V^{a_2, a_1}\{-a_1-a_2, -a_1-a_2\}. \]
   This representation will play an important role in the following, and we shall write it as $D^{a_1, a_2}$.
  \end{note}

%%%%%%%%%%%%%%%%%%%%%%%%%%%%%%%%%%%%%

 \subsection{Branching laws}\label{ssec:branching}

  The restriction of $G$-representations to $H$ is described by a branching law, which is equivalent to the usual branching law for $\GL_2 \subset \GL_3$ (see e.g.~\cite[Theorem 8.1.1]{goodmanwallach09}). The statement we need is the following:
  \begin{propqed}
   \label{cor:specialsub}
    The representation $D^{a_1, a_2}\{b_1, b_2\}$ has a non-zero $Q_H^0$-invariant vector if and only if $0 \le b_i \le a_i$. In this case, there is a unique such vector up to scaling, and it is the highest-weight vector of the unique $H$-subrepresentation isomorphic to $W^n\{-n, -n\}$, where $n = a_1 + a_2 - b_1 - b_2$.
  \end{propqed}

  \begin{remark}
   The representations $W^n\{-n, -n\}$ are important since they are the coefficient systems for which we can construct motivic Eisenstein classes; see \cref{sect:LEconstruct} below.
  \end{remark}

  We fix normalisations for these $Q_H^0$-invariant vectors using \cref{lem:openorbit}. Let $u \in N_G(\ZZ[1/D])$ be a choice of element satisfying the conclusion of that lemma.

  \begin{proposition}
  \label{prop:brabrs}
   Suppose $0 \le r \le a, 0 \le s \le b$ are integers, and let $d^{[a,b]}$ be a choice of highest-weight vector of $D^{a,b}$. Then there exists a unique vector
   \[ \br^{[a,b,r,s]} \in \left(D^{a, b}\{r, s\} \right)^{Q_H^0} \]
   with the following property: the projection of $u^{-1} \cdot \br^{[a,b,r,s]}$ to the highest-weight space of $D^{a,b}\{r, s\}$ is $d^{[a,b]}\{r, s\}$.
  \end{proposition}

  \begin{proof}
   Let $\lambda$ be the highest weight of $D^{[a,b]}\{r, s\}$. We use the Borel--Weil presentation of $D^{[a,b]}\{r, s\}$: it is isomorphic to the space of polynomial functions on $G$ which transform via $\lambda$ under left-translation by $\bar{B}_G$. This space has a canonical highest-weight vector $f^{\mathrm{hw}}$, whose restriction to the big Bruhat cell is given by $f^{\mathrm{hw}}(\bar{n} t n) = \lambda(t)$.

   If $f^H$ denotes the polynomial corresponding to $\br^{[a,b,r,s]}$, then $f^H$ must transform via $\lambda$ under left-translation by $\bar{B}_G$, and trivially under right-translation by $Q_H^0$. Since $\bar{B}_G u^{-1} Q_H^0$ is open, we must have $f^H(u^{-1}) \ne 0$, so we can normalise such that $f^H(u^{-1}) = 1$.

   Since projection to the highest-weight subspace is proportional to evaluation at the identity, and both $u^{-1} f^H$ and $f^{\mathrm{hw}}$ take the value 1 at the identity, this shows that $u^{-1} \cdot f^H$ has the same highest-weight projection as $f^{\mathrm{hw}}$.
  \end{proof}

  For $F$ an extension of $E$, we write $D^{a,b}_F\{r, s\}$ for the base-extension of $D^{a,b}_F\{r, s\}$ to $F$, which is an irreducible representation of $G_{/F}$. If $F = E_w$ for a prime $w \mid D$, then $G$ is a Chevalley group (a reductive group scheme) over $\cO_{E, w}$, so we have the notion of \emph{admissible} $\cO_{E, w}$-lattices in the $E_w$-vector space $D^{a,b}\{r, s\} \otimes_E E_w$; see \cite{lin92} for an overview. We are chiefly interested in the \emph{maximal} admissible lattice, which we shall denote by $D^{a,b}_{\cO_{E, w}}\{r, s\}$.

  \begin{proposition}
   The vector $\br^{[a,b,r,s]}$ lies in $D^{a,b}_{\cO_{E, w}}\{r, s\}$ for all primes $w \nmid D$.
  \end{proposition}

  \begin{proof}
   As shown in \cite[\S2.3]{lin92}, the maximal lattice can be constructed explicitly via the Borel--Weil description of $D^{[a,b]}\{r, s\}$: it is the intersection of $D^{[a,b]}_{E_w}\{r, s\} \subset E_w[G]$ with the integral coordinate ring $\cO_{E_w}[G]$. So we must show that the polynomial $f^H$ in \cref{prop:brabrs} lies in $\cO_{E_w}[G]$.

   Let $\mathbf{F}_w$ be the residue field of $E_w$. Then $f^{H}$ is regular on $G_{/E_w}$; and it is also regular on a dense open subscheme of $G_{/\mathbf{F}_w}$. So it is regular on a subset of $G_{/\cO_{E, w}}$ of codimension $\ge 2$. Since $G_{/\cO_{E, w}}$ is smooth, it is a normal scheme. It follows that $f^H$ is regular everywhere on $G_{/\cO_{E, w}}$ (see e.g.~Stacks Project tag 031T).
  \end{proof}

%%%%%%%%%%%%%%%%%%%%%%%%%%%%%%%%%%%%%%%%%%%%%%%%%%%%%%%%%%%%%%%%%%%%%%%%%%%%%%%%
\section{Shimura varieties}
\label{sect:shimura}
%%%%%%%%%%%%%%%%%%%%%%%%%%%%%%%%%%%%%%%%%%%%%%%%%%%%%%%%%%%%%%%%%%%%%%%%%%%%%%%%

 \subsection{The Shimura varieties $Y_G$ and $Y_H$}

  \subsubsection{The Shimura variety $Y_G$}

   Let $\SS = \Res_{\CC/\RR}\GG_m$, and consider the homomorphism
   \[
    h:\SS\rightarrow G_{/\RR},\qquad h(z) = (\tfrac{1}{a^2 + b^2}\left(\smallmatrix a & & b \\ & z & \\ -b & & a \endsmallmatrix\right), \tfrac{1}{a^2 + b^2}), \quad z=a+ib\in \SS(\RR) = \CC^\times.
   \]
   We write $\cX_G$ for the space of $G(\RR)$-conjugates of $h$; we can identify $\cX_G$ as the unbounded Hermitian symmetric domain
   \[
    \{(z,w) \in \CC\times\CC \ : \ \Im(z) - w\bar w > 0 \}, \ \ (g,\nu) \cdot h \mapsto (a/c,b/c)\quad\text{where}\quad g\cdot\left [ \smallmatrix i \\ 1 \\ 1 \endsmallmatrix\right ] = \left[\smallmatrix a \\ b \\ c \endsmallmatrix\right].
   \]
   Then $(G,h,\cX_G)$ is a Shimura datum.

   \begin{remark}
    Our choice of Shimura datum is a little non-standard; it is more common to use the alternative Shimura datum defined by $h'(z) = h(1/\bar{z})$, which is the image of $h$ under the automorphism of $G$ given by $(g, \nu) \mapsto (\nu^{-1}g, \nu^{-1})$. However, using $h$ rather than $h'$ gives simpler formulae for motivic Eisenstein classes. Compare \cite[Remark 5.1.2]{LSZ17}.
   \end{remark}

   The reflex field of this Shimura datum is $E$ (viewed as a subfield of $\CC$ via our chosen identification of $E\otimes \RR$ with $\CC$). We let $Y_G$ be the canonical model over $E$ of the Shimura variety associated with this datum. For any open compact subgroup $K\subset G(\Af)$ we let $Y_G(K) = Y_G / K$ be the quotient by $K$; this is a quasi-projective variety over $E$. If $K$ is sufficiently small, it is smooth (it suffices to take $K$ to be \emph{neat} in the sense of \cite{pink90}; see \cite[\S 2.3]{gordon92}). We recall that the $\CC$-points of $Y_G(K)$ have a natural description as
   \[
    Y_G(K)(\CC) = G(\QQ)\bs [\cX_G \times G(\Af)/K].
   \]

  \subsubsection{The Shimura variety $Y_H$}

  The homomorphism $h$ factors as $\iota \circ h_H$, where $h_H: \SS \to H_{/\RR}$ is the Shimura datum
  \[ z = a+ ib \mapsto \left( \tfrac{1}{a^2 + b^2}\stbt{a}{b}{-b}{a}, \bar{z}^{-1}\right). \]
  We let $\cX_H$ be the $H(\RR)$-conjugacy class of $h_H$.  Then $(H,h,\cX_H)$ is also a Shimura datum, and its reflex field is also $E$. We let $Y_H$ be the canonical model over $E$ of the associated Shimura variety. For an open compact $K'\subset H(\Af)$, the $\CC$ points of the quasi-projective variety $Y_H(K')$ are naturally described as
  \[
   Y_H(K')(\CC) = H(\QQ)\bs[\cX_H\times H(\Af)/K'].
  \]

  \subsubsection{Functoriality} The inclusion $\iota:H\into G$ induces an $E$-morphism $Y_H \rightarrow Y_G$. In particular, if $K\subset G(\Af)$ and $K'\subset H(\Af)$ are such that $K' \subset K\cap H(\Af)$, then there is a finite morphism of $E$-varieties $Y_H(K')\rightarrow Y_G(K)$ that on $\CC$-points is just the map
  \[
   H(\QQ)\bs[\cX_H\times H(\Af)/K'] \rightarrow G(\QQ)\bs [\cX_G \times G(\Af)/K], \ \ \ [h',h_f] \mapsto [\iota\circ h', \iota(h_f)].
  \]

  We also have the projection map $\pi: H \to \GL_2$ (forgetting $z$). The composite $\pi \circ h$ is a Shimura datum for $\GL_2$, which coincides with the one used in \cite[\S 5.1]{LSZ17}; again, this differs from the ``standard'' Shimura datum by an automorphism of $\GL_2$.

 \subsection{The component groups of $Y_G$ and $Y_H$}

  The set $\pi_0(Y_G)$ of connected components of $Y_G$ can be described as follows.
  Let $\mu = \overline{\det}/\nu: G\rightarrow \Res_{E/\QQ}(\GG_m)$, so that the composite $\mu \circ h$ is given by $z \mapsto z^{-1}$.

  Then the map
  \[
   Y_G(K)(\CC)\stackrel{\pi_0}{\rightarrow} E^\times\bs (E\otimes\Af)^\times/\mu(K), \ \ \pi_0([h,g_f])\mapsto \mu(g_f),
  \]
  identifies the set of geometrically connected components $\pi_0(Y_G(K))$ of $Y_G(K)$ with $E^\times\bs (E\otimes\Af)^\times/\mu(K)$. So
  \[
   \pi_0(Y_G) = E^\times\bs (E\otimes\Af)^\times.
  \]
  The action of $\Gal(\bar{E}/E)$ on $\pi_0(Y_G)$ can be described by the reciprocity law: if
  \[
   \Art_E: E^\times\bs (E^\times\otimes\Af)^\times \isoarrow \Gal(\bar{E}/E)^{\mathrm{ab}}
  \]
  is the Artin reciprocity map of class field theory, normalized so that geometric Frobenius elements are mapped to uniformizers, then the map $\pi_0(Y_G) \cong E^\times \bs (E^\times\otimes\Af)^\times$ is $\Gal(\bar{E}/E)$-equivariant if we let $\sigma \in \Gal(\bar{E}/E)$ act on $E^\times \bs (E^\times\otimes\Af)^\times$ as multiplication by $\Art_E(\sigma)^{-1}$. The same analysis applies also to $Y_H$ in place of $Y_G$, since $\iota$ identifies $H / [H, H]$ with $G / [G, G]$.

  We can regard $G$ as a subgroup of $G \times \Res_{E/\QQ}\GG_m$, via the map $(\mathrm{id}, \mu)$. If $K$ is any open compact in $G(\Af)$, and $K[\fm] = \{k \in K : \mu(k) = 1 \bmod \fm\}$ for an ideal $\fm$ of $E$, then this gives an open-and-closed embedding
  \begin{equation}
   \label{eq:component}
   Y_G(K[\fm]) \into Y_G(K) \times_{\Spec E} \Spec E[\fm].
  \end{equation}
  Note that this intertwines the action of a Hecke operator $[K[\fm] g K[\fm]]$ on the left-hand side with $[K g K] \times \Art_E(\mu(g))^{-1}$ on the target.

 \subsection{Sheaves corresponding to algebraic representations}
 \label{sect:ancona}

  Let $\sG$ temporarily denote any of  the three groups $\left\{ \GL_2,\, H,\, G\right\}$, and let $F$ be a number field. As in \cite[\S 6]{LSZ17}, we can define a category of \emph{$\sG(\Af)$-equivariant relative Chow motives} on the infinite-level Shimura variety $Y_{\sG}$, with coefficients in $F$; an object of this category is a collection $\sV = (\sV_U)_U$ of $F$-linear relative Chow motives over $Y_{\sG}(U)$ for all sufficiently small open compacts $U \subset \sG(\Af)$, satisfying compatibilities under pullback and translation by $\sG(\Af)$. We denote this category by $\operatorname{CHM}_F(Y_G)^{\sG(\Af)}$. If $\sV$ is an object of this category, its motivic cohomology
  \[
   H^*_{\mot}(Y_{\sG}, \sV) = \varinjlim_UH^*_{\mot}(Y_{\sG}(U), \sV_U),
  \]
  is naturally a smooth $F$-linear (left) representation of $\sG(\Af)$.

  \begin{theorem}[{\cite[Theorem 8.6]{ancona15}}]
   There is an additive functor
   \[ \Anc_{\sG}: \operatorname{Rep}_F(\sG) \to \operatorname{CHM}_F(Y_{\sG})^{\sG(\Af)} \]
   with the following properties:
   \begin{enumerate}[(i)]
    \item $\Anc_{\sG}$ preserves tensor products and duals.

    \item if $\nu$ denotes the multiplier map $\sG \to \mathbf{G}_m$, then $\Anc_{\sG}(\nu)$ is the Lefschetz motive $F(-1)[-1]$, where $[-1]$ denotes that the $\sG(\Af)$-equivariant structure is twisted by the character $\|\nu\|^{-1}$.

    \item for any prime $v$ of $F$ and $\sG$-representation $V$, the $v$-adic realisation of $\Anc_{\sG}(V)$ is the equivariant \'etale sheaf associated to $V \otimes_F F_v$, regarded as a left $G(\Qp)$-representation where $p$ is the prime below $v$.
   \end{enumerate}
  \end{theorem}

  We shall always take the coefficient field $F$ to be $E$, and frequently drop it from the notation.

  \begin{propqed}[{\cite[Corollary 9.8]{torzewski18}}]
   \label{prop:branching}
  There is a commutative diagram of functors
   \[
   \begin{tikzcd}
   \operatorname{Rep}(G) \rar["\Anc_G"] \dar["\iota^*" left]
   & \operatorname{CHM}(Y_G)^{G(\Af)} \dar["\iota^*"]\\
   \operatorname{Rep}(H) \rar["\Anc_H"] & \operatorname{CHM}(Y_H)^{H(\Af)}
   \end{tikzcd}
   \]
   where the left-hand $\iota^*$ denotes restriction of representations, and the right-hand $\iota^*$ denotes pullback of relative motives.
  \end{propqed}

%%%%%%%%%%%%%%%%%%%%%%%%%%%%%%%%%%%%%%%%%%%%%%%%%%%%%%%%%%%%%%%%%%%%%%%%%%%%%%%%%%
\section{Construction of the unitary Eisenstein classes}
\label{sect:UEconstruct}
%%%%%%%%%%%%%%%%%%%%%%%%%%%%%%%%%%%%%%%%%%%%%%%%%%%%%%%%%%%%%%%%%%%%%%%%%%%%%%%%%%

 \subsection{Pushforwards in motivic cohomology}
  \label{sect:UEconstruct1}

  Let $0 \le r \le a$, $0 \le s \le b$ be integers. We use script letters $\sV^{a,b}$, $\sD^{a,b}\{r, s\}$ etc for the images of the corresponding algebraic representations under Ancona's functor. For $n \ge 0$, we write $\sH^n = \Anc_{H}(W^{n}\{-n, -n\})$. Taking $n =a+b-r-s$, \cref{prop:branching} gives us maps of equivariant relative Chow motives on $Y_{H}$
  \begin{equation}
   \label{eq:motivemaps}
   \sH^n \into \iota^*\left(\sD^{[a,b]}\{r, s\}\right),
  \end{equation}
  where the latter map is normalised to send the highest-weight vector of $W^n\{-n, -n\}$ to the vector $\br^{[a,b,r,s]} \in D^{a,b}\{r,s\}$ of \cref{prop:brabrs}. If we fix an open compact subgroup $U \subset G(\Af)$, and an element $g \in G(\Af) / U$, then we have a finite map
  \[ \iota_{gU}: Y_H(H \cap gUg^{-1}) \longrightarrow Y_G(U), \]
  given by the composite of $\iota: Y_H(H \cap gUg^{-1}) \to Y_G(gUg^{-1})$ and translation by $g$. Since motivic cohomology is covariantly functorial (with a shift in degree) for finite morphisms of smooth varieties, we obtain from \eqref{eq:motivemaps} a homomorphism
  \[
   \iota_{gU, \star}^{[a,b,r,s]}: H^1_{\mot}\left(Y_H(H \cap gUg^{-1}), \sH^n(1)\right)\longrightarrow
   H^3_{\mot}\left(Y_G(U), \sD^{a,b}\{r, s\}(2)\right)
  \]
  for each $U$. Exactly as in \cite[\S 8.2]{LSZ17}, we have:

  \begin{propqed}
   \label{prop:iabrs}
   Let $\vol$ denote a choice of $E$-valued Haar measure on $H(\Af)$. Then there is a unique map
   \[ \iota_{\star}^{[a,b,r,s]}: H^1_{\mot}\left(Y_H, \sH^n\right)\, \otimes_E\,
   \cH(G(\Af); E) \rightarrow H^3_{\mot}\left(Y_G, \sD^{a,b}\{r, s\}(2)\right)\]
   characterised as follows: if $U$ is an open compact in $G$, $g \in G(\Af)$, and $x \in H^1_{\mot}\left(Y_H(V), \sH^n(1)\right)$ where $V = H \cap g U g^{-1}$, then we have
   \[ \iota_{\star}^{[a,b,r,s]}(x \otimes \ch(gU)) = \vol(V) \cdot \iota_{gU, \star}^{[a,b,r,s]}(x).\qedhere\]
  \end{propqed}

  \begin{remark}
   The proof that this map is well-defined ultimately reduces to the compatibility of pushforward and pullback in Cartesian diagrams; it therefore carries over to the general setting of \emph{Cartesian cohomology functors} for $G$ and $H$, in the sense of \cite{loeffler-spherical}. For a careful proof of the well-definedness using this formalism, see \cite[Proposition 5.9]{grahamshah20}.
  \end{remark}

 \subsection{Eisenstein classes and the unitary Eisenstein map}
  \label{sect:LEconstruct}

  \begin{definition}[Siegel, Beilinson]
   For $k \in \ZZ_{\ge 0}$, the \emph{motivic Eisenstein symbol} of weight $k$ is the $\GL_2(\Af)$-equivariant map
   \[
    \cS_{(0)}(\Af^2, E) \to H^1_{\mot}\left(Y_{\GL_2},\sH^k(1)\right), \qquad\phi \mapsto \Eis^k_{\mot, \phi},
   \]
   described in \cite[Theorem 7.2.2]{LSZ17}. Here $\cS_{(0)}$ signifies $\cS$ if $k \ge 1$ and $\cS_0$ if $k = 0$.
  \end{definition}

  \begin{remark}
   This map can be characterised via its residue at $\infty$, or via its composite with the de Rham realisation functor; see \emph{loc.cit.} for explicit formulae. When $k = 0$ and $\phi$ is the characteristic function of $(\alpha, \beta) + \hat\ZZ^2$, for $\alpha, \beta \in \QQ/\ZZ$ not both zero, we have $H^1_{\mot}\left(Y_{\GL_2},\sH^k(1)\right) = H^1_{\mot}\left(Y_{\GL_2},E(1)\right) = \cO(Y_{\GL_2})^\times \otimes E$, and $\Eis^k_{\mot, \phi}$ is the Siegel unit $g_{\alpha, \beta}$ in the notation of \cite{kato04}.
  \end{remark}

  Composing the Eisenstein symbol with pullback along the projection $Y_H \to Y_{\GL_2}$ defines an $H(\Af)$-equivariant map
  \( \cS_{(0)}(\Af^2; E)\to H^1_{\mot}\left(Y_H,  \sH^k(1)\right) \) which we denote by the same symbol.

  \begin{definition}\label{def:uEis}
   We define the \emph{unitary Eisenstein map}
   \[
    \UE^{[a,b,r,s]}:
    \cS_{(0)}(\Af^2;E)\otimes\cH(G(\Af);E) \rightarrow
    H^3_{\mot}\left(Y_G, \sD^{a,b}\{r, s\}(2)\right)
   \]
   by $\UE^{[a,b,r,s]}(\phi \otimes \xi) = \iota_*^{[a, b, r, s]}\left( \Eis^{a+b-r-s}_{\mot, \phi} \otimes \xi\right)$, where $\iota_*^{[a,b,r,s]}$ is the map of \cref{prop:iabrs}.
  \end{definition}

  By construction, this map is $G(\Af) \times H(\Af)$-equivariant in the sense of \cref{def:equivariant}.

 \subsection{Choices of the local data}
  \label{sect:localdata}

  We shall now fix choices of the input data to the above map $\UE^{[a,b,q,r]}$, in order to define a collection of motivic cohomology classes satisfying appropriate norm relations (a ``motivic Euler system''). We shall work with arbitrary (but fixed) choices of local data at the bad primes; it is the local data at good primes which we shall vary, depending on a choice of a parameter $\fm$.

  \begin{definition}
   Let $S$ be a finite set of (rational) primes, containing all primes dividing $2d$. Let $\cR$ denote the set of square-free ideals $\fm$ of $\cO$, coprime to $S$, with the following  property: for each prime $\ell = w\bar{w}$ split in $E$, at most one of $\{w, \bar{w}\}$ divides $\fm$.
  \end{definition}

  We choose an arbitrary element $\delta_S \in \cS_{(0)}(\QQ_S, E) \otimes \cH(G(\QQ_S), E)$, and an open compact subgroup $K_{G, S} \subset G(\QQ_S)$ fixing $\delta_S$. We use these to define a collection of elements $(\delta[\fm])_{\fm \in \cR}$ of \(\cS_{(0)}\left(\Af^2, E\right) \otimes \cH(G(\Af), E)\), given by $\delta[\fm] = \delta_S \cdot \bigotimes_{\ell \notin S} \delta_\ell[\fm]$, where:
  \begin{itemize}
   \item if $\ell \notin S$ and $(\ell, \fm) = 1$, then $\delta_\ell[\fm]$ is the unramified element $\ch(\Zl^2) \otimes \ch(G^0_\ell)$;
   \item if $\fm$ is divisible by some prime $w \mid \ell$, then $\delta_\ell[\fm]$ is the element $\delta_w = n_w \phi_{1, 2} \otimes \xi_w$ defined in \cref{def:tamedata}.
  \end{itemize}

  Thus $\xi[\fm]$ is preserved under right-translation by the open compact subgroup $K_G[\fm] = K_{G, S} \times \{ g \in G(\hat{\ZZ}^S) : \mu(g) = 1\bmod \fm\}$ of $G(\Af)$. Moreover, if we suppose that $\delta_S \in \cI(G_S / K_{G, S}, \ZZ)$, then for all $\fm \in \cR$ we have $\delta[\fm] \in \cI\left(G(\Af) / K_{G}[\fm], \ZZ[1/\Nm(\fm)]\right)$.

 \subsection{The ``motivic Euler system''}

  \begin{definition}
   \label{def:cZ}
   We set
   \[ \cZ_{\mot, \fm}^{[a,b,r,s]}(\delta_S) \coloneqq  \UE^{[a,b,r,s]}\left(\delta[\fm]\right) \in H^3_{\mot}\left(Y_G[\fm], \sD^{a,b}\{r, s\}(2)\right). \]
  \end{definition}

  Note that this depends $(H_S \times G_S)$-equivariantly on $\delta_S$ (for fixed $\fm$ and $(a,b,r,s)$). We shall frequently omit $\delta_S$ from the notation.

  \begin{remark}
   Note that $Y_G[\fm]$ has a smooth integral model over $\cO[S^{-1}, \Nm(\fm)^{-1}]$, which we denote by $\cY_G[\fm]$. One verifies easily that the relative motive $\sD^{a,b}\{r, s\}$ and the cohomology class $\cZ_{\mot, \fm}^{[a,b,r,s]}(\phi_S, \xi_S)$ both have natural extension to this smooth model.
  \end{remark}

  \begin{theorem}
   \label{thm:tamenorm}
   Let $\fm, \fn \in \cR$ with $\fm \mid \fn$. If $\pr_{\fm}^{\fn}$ denotes the natural map $Y_G[\fn] \to Y_G[\fm]$, then we have
   \[ \left(\pr_{\fm}^{\fn}\right)_\star\left(\cZ_{\mot, \fn}^{[a,b,r,s]}\right) = \bigg(\prod_{w \mid \tfrac{\fn}{\fm}} \cP_w'(1) \bigg) \cdot \cZ_{\mot, \fm}^{[a,b,r,s]},
   \]
   where $\cP_w'(1)$ is the Hecke operator appearing in \cref{thm:absnorm2}.
  \end{theorem}

  \begin{proof}
   It clearly suffices to assume that $\fn = \fm w$ for a prime $w$. The result is now a direct consequence of \cref{thm:absnorm2}, with $\ell$ the prime below $w$. Fixing the input data away from the prime $\ell$, we can regard $\UE^{[a,b,r,s]}$ as an $H_\ell \times G_\ell$-invariant map $\cS(\Ql^2) \times \cH_{G, \ell} \to V$ where $V$ denotes the representation
   \[ V = \varinjlim_{U_\ell \subset G_\ell} H^3_{\mot}\left(Y_G(K_G^{(\ell)}[\fm] \times U_\ell), \sD^{a,b}_E\{r, s\}(2)\right).\]

   We note that this $V$ does satisfy the auxiliary hypothesis on the action of the torus $A$: as a representation of $A(\Ql)$, $V$ is a direct sum of eigenspaces associated to characters of $\Ql^\times$ of the form $x \mapsto |x|^{n} \chi(x)$ with $\chi$ of finite order and $n = a+b-r-s \ge 0$. Thus $z_A(\ell) + \ell$ is bijective on $V$. The corollary now gives an equality between two values of this $H_\ell \times G_\ell$-invariant map on different input data, and these are precisely the local input data used to define $\cZ_{\mot, \fm}^{[a,b,r,s]}$ and the pushforward of $\cZ_{\mot, \fn}^{[a,b,r,s]}$.
  \end{proof}

  We can give an alternative interpretation of these classes via \cref{eq:component}. We denote by $\Xi_{\mot, \fm}^{[a,b,r,s]}(\delta_S)$ the pushforward of $\cZ(\dots)$ to an element of $H^3_{\mot}\left(Y_G[1]_{E[\fm]}, \sD^{a,b}\{r, s\}(2)\right)$; again, we frequently omit $\delta_S$.

  \begin{definition}
   For $w\nmid \fm$ a prime of $E$, let $\sigma_w$ denote the arithmetic Frobenius at $w$, as an element of $\Aut(E[\fm] / E)$.
  \end{definition}

  One checks that \eqref{eq:component} intertwines the action of $\cP_w'(1)$ on the source with $\cP_w'(\sigma_w^{-1})$ on the target, so we can write the norm-compatibility relation as
  \begin{equation}
   \label{eq:tamenorm2}
   \operatorname{norm}_{E[\fm]}^{E[\fn]}\left(\Xi_{\mot, \fn}^{[a,b,r,s]}\right) = \bigg(\prod_{w \mid \tfrac{\fn}{\fm}} \cP'_w(\sigma_w^{-1}) \bigg) \cdot \Xi_{\mot, \fm}^{[a,b,r,s]}.
  \end{equation}

 \subsection{\'Etale realisation and integrality}

  It would be desirable to have an ``integral'' version of this theory, with coefficients in $\cO$-modules, but this appears to be difficult for general coefficients (we do not know if the functors $\Anc_{\sG}(-)$ can be defined integrally). So we shall instead work with the $p$-adic \'etale realisation, for a fixed prime $p$. In this section, we will fix values of $[a,b,r,s]$ and omit them from the notation.

  Let $p$ be a (rational) prime, and $\fp \mid p$ a prime of $E$. We define
  \[ \cZ_{\et, \fm}(\delta_S) \coloneqq r_{\mathrm{et}}\left(\cZ_{\mot, \fm}(\delta_S)\right) \in H^3_{\et}\left(Y_G[\fm], \sD^{a,b}_{E_\fp}\{r, s\}(2)\right)\]
  where $\sD^{a,b}_{E_\fp}$ is the \'etale sheaf of $E_{\fp}$-vector spaces corresponding to $D^{a,b} \otimes_E E_{\fp}$, and similarly $\Xi_{\et, \fm}(\delta_S)$.

   For simplicity, we assume here that $p \notin S$ (similar, but more complicated, statements can be formulated if $p \in S$). If $c$ is a prime, coprime to $6\fm$ and not in $S$, we shall write $\langle c \rangle$ for the action of $z_A(\varpi_c)$, where $\varpi_c$ is a uniformizer of $\QQ_c$. We extend this multiplicatively to all integers $c > 1$ coprime to $6 \Nm(\fm)S$. Then we define
  \begin{gather*}
   {}_c\cZ_{\et, \fm}(\delta_S) \coloneqq (c^2 - c^{-n}\langle c \rangle) \cdot \cZ_{\et, \fm}(\delta_S),\\
   {}_c\Xi_{\et, \fm}(\delta_S) \coloneqq (c^2 - c^{-n} \langle c \rangle \sigma_c) \cdot \Xi_{\et, \fm}(\delta_S),
  \end{gather*}
  where $\sigma_c$ in the latter formula is the arithmetic Frobenius. (These definitions are consistent with one another, since the map $Y_G[\fm] \to Y_G[1] \mathop{\displaystyle{\times}}_{E} E[\fm]$ intertwines $\langle c \rangle$ on the source with $\langle c \rangle \sigma_c$ on the target.)

  \begin{definition}
   We write $D^{a,b}_{\cO_{E, \fp}}$ for the maximal admissible $\cO_{E,\fp}$-lattice in $D^{a,b} \otimes E_{\fp}$, and $\sD^{a,b}_{\cO_{E, \fp}}$ for the corresponding \'etale sheaf.
  \end{definition}

  \begin{proposition}
  \label{prop:intetale}
   Suppose $\delta_S \in \cI(G_S / K_{G, S}, \cO_{E, (\fp)})$. Then, for every $\fm \in \cR$ coprime to $p$ and every $c > 1$ coprime to $6\fm S$, the classes ${}_c\cZ_{\et, \fm}(\delta_S)$ and ${}_c\Xi_{\et, \fm}(\delta_S)$ lie in the image of the cohomology of the integral coefficient sheaf $\sD^{a,b}_{\cO_{E,\fp}}\{r, s\}$.
  \end{proposition}

  \begin{proof}
   Since the local terms $\delta_\ell[\fm]$ for primes $\ell \mid \Nm(\fm)$ are integral away from $\ell$ by construction, we can replace $S$ with $S \cup \{ \ell :\ell \mid \Nm(\fm) \}$, and thus reduce to the case $\fm = 1$. Let us abbreviate $K_G[1]$ simply by $K_G$.

   We may also suppose $\delta_S = \phi_S \otimes \ch(g K_{G, S})$ is a primitive integral element in the sense of \cref{def:integral}. Let $V_S = \operatorname{stab}_{H_S}(\phi_S) \cap g K_{G, S} g^{-1}$, and write $V = V_S \cdot H(\hat\ZZ^S)$. By assumption, the values of $\phi_S$ land in $C \cdot \cO_{E,(\fp)}$, where $C = \frac{1}{\vol V_S}$.

   We note that the Eisenstein class $\Eis^n_{\et, \phi}$ (the \'etale realisation of $\Eis^n_{\mot, \phi}$) has an integral variant ${}_c\Eis^n_{\et, \phi}$, taking values in the cohomology of $Y_H(V)$ with values in the \emph{minimal} admissible lattice in $\sH^n$. The branching map $\br^{[a,b,r,s]}$ maps this into the pullback of the \emph{maximal} admissible lattice in $\sD^{a,b}\{r, s\}$ (compare \cite[Proposition 4.3.5]{LSZ17}). Since $C^{-1} \phi$ is $\cO_{E, (\fp)}$-valued, we conclude that the image of $C^{-1} {}_c\Eis^n_{\et, \phi}$ under pushforward to $H^3_{\et}(Y_G(g K_G g^{-1}),\sD^{a,b}\{r, s\}(2))$ lifts (canonically) to the cohomology of the integral coefficient sheaf. Since $C^{-1} = \vol_{H}(V)$ is the normalising factor in the definition of the unitary Eisenstein class, this shows that ${}_c\cZ_{\et, \fm}(\delta_S)$ lifts to the integral cohomology, as required.
  \end{proof}

%%%%%%%%%%%%%%%%%%%%%%%%%%%%%%%%%%%%%%%%%%%%%%%%%%%%%%%%%%%%%%%%%%%%%%%%%%%%%%%%
\section{Norm relations at $p$}
 \label{sect:normrel}
%%%%%%%%%%%%%%%%%%%%%%%%%%%%%%%%%%%%%%%%%%%%%%%%%%%%%%%%%%%%%%%%%%%%%%%%%%%%%%%%

 We now consider norm-compatibility relations in the ``$p$-direction''. We let $p$ and $\fp$ be as in the previous section, and we add the additional assumption that $c$ is coprime to $p$.

 \subsection{Choice of local data}

  \begin{definition}
    Let $\tau = \left[\mathrm{diag}(p^2,\, p,\, 1),\, p^2\right] \in T_G(\Qp)$. For $t \geq 1$, define
   \begin{itemize}
    \item $K_{G_p}(p^t) = \left\{ g\in G(\Zp):\, \tau^{r}g\tau^{-r}\in G(\Zp)\quad \text{and}\quad g\pmod{p^t}\in N_G(\ZZ/p^t)\right\}$.
    \item $\xi_{p, t} =\ch\left( u \tau^t \cdot K_{G_p}(p^t)\right)$, where $u$ is an element of $G(\Zp)$ satisfying the conditions of \cref{lem:openorbit}.
    \item $\phi_{p, t} = \ch( (p^{2t} \Zp) \times (1 + p^{2t}\Zp))$ if $t \ge 1$, and $\ch(\Zp^2)$ if $t = 0$.
    \item finally, $n_{p, t}$ denotes the index in $H(\Zp)$ of the subgroup
      \[ V_{p, t} = K_{H_p, 1}(p^{2t}) \cap u\tau^t K_{G_p}(p^{t}) (u\tau^t)^{-1}, \]
      given for $t \ge 1$ by
      \[ n_{p,t} =
       \begin{cases}
        p^{6t-4}(p - 1)^3(p+1) & \text{if $p$ split}\\
        p^{6t-4}(p - 1)^2(p+1)^2 & \text{if $p$ inert}.
       \end{cases}
      \]
   \end{itemize}
   We then set $\delta_{p, t} = n_{p, t}\phi_{p, t} \otimes \xi_{p, t} \in \cI(G^0_p / K_{G_p}(p^t), \ZZ)$.
  \end{definition}

  \begin{remark}
   Explicitly, we have
   \[ K_{G_p}(p^t)=\left\{ (g,\nu)\in G(\Zp): g= \begin{pmatrix} a & \star & \star\\ b & c & \star\\ d & e & f\end{pmatrix},
   \begin{gathered}
   a\equiv c\equiv f\equiv 1\bmod{p^t},\\  b\equiv e\equiv 0\bmod{p^t},\\ d\equiv0\bmod{p^{2t}}. \, \end{gathered}\right\}.\]
   (These conditions also entail $\nu = 1 \bmod p^t$.) The subgroup $V_{p, t}$ consists of all $(\stbt a b c d, z) \in H(\Zp)$ with $c = 0, d = 1 \bmod p^{2t}$, $z = 1 \bmod p^t$, and $b$ satisfying a certain somewhat messy congruence modulo $p^{2t}$  (whose precise form depends on the choice of $u$).
  \end{remark}

  Now let us choose arbitrary $\delta_S\in \cI(G / K_{G, S}, E)$ as before. For $t \ge 0$, and $\fm \in \cR$ coprime to $p$, we can define $\delta[\fm, p^t] = \delta_{S} \cdot \delta_{p, t} \cdot \prod_{\ell \notin S \cup \{p\}} \delta_\ell[\fm]$, so that $\xi[\fm, p^t]$ is fixed by the right action of the group $K_G[\fm, p^t] = K_{G, S} \cdot K_{G_p}(p^t) \cdot \{ g \in G(\hat\ZZ^S): \mu(g) = 1 \bmod \fm\}$.

  \begin{definition}
   With the above notations, we set
   \[
    \cZ_{\mot, \fm, p^t}^{[a,b,r,s]}(\delta_S)\ \coloneqq\ p^{(r+s)t} \UE^{[a,b,r,s]}\left(\delta[\fm, p^t]\right)\ \in\ H^3_{\mot}\left(Y_G(K_G[\fm, p^t]), \sD^{a,b}\{r, s\}(2)\right).
   \]
  \end{definition}

  Since this definition is a special case of \cref{def:cZ}, these elements satisfy the norm-compatibility in $\fm$ of \cref{thm:tamenorm}; and it also clearly depends $(G(\QQ_{S}) \times H(\QQ_{S}))$-equivariantly on the test data $\delta_S$ at the bad primes. For the rest of this section we regard $\delta_S$ as fixed, and drop it from the notation.

  Similarly, we can introduce $p$-level structure to the classes $\Xi_{\mot, \fm}$ as follows. Let $Y_{\Ih}$ denote the Shimura variety of level $K_{G, S} \cdot \mathrm{Ih}_p \cdot G(\hat\ZZ^{S \cup \{p\}})$, where $\mathrm{Ih}_p = \{ g \in G(\Zp): g \bmod p \in B_G(\mathbf{F}_p)\}$ is the upper-triangular Iwahori\footnote{We use the abbreviation ``Ih'' rather than ``Iw'' to avoid confusion with Iwasawa.} at $p$. Then we have a natural map
  \[ Y_G(K_G[\fm, p^t]) \longrightarrow Y_{\Ih} \mathop{\times}_E E[\fm p^t]. \]
  We let
  \[ \Xi^{[a,b,r,s]}_{\mot, \fm, p^t} \in H^3_{\mot}\left(Y_{\Ih} \mathop{\times}_E E[\fm p^t], \sD^{a,b}\{r, s\}(2)\right) \]
  be the image of $\cZ^{[a,b,r,s]}_{\mot, \fm, p^t}$ under pushforward along this map.

 \subsection{Norm-compatibility in $t$}
  \label{sect:wildnorm}

  We now observe that these classes satisfy norm-compatibility in $t$.

  \begin{definition}
   Let $\mathcal{U}_p'$ denote the Hecke operator acting on $Y_G(K_G[\fm, p^t])$, with coefficients in $\sD^{a,b}\{r, s\}$, given by $p^{(r + s)} \left[ K_{G_p}(p^t)\tau^{-1} K_{G_p}(p^t)\right]$.
  \end{definition}

  This operator preserves the integral \'etale cohomology, because $p^{r+s}$ bounds the denominator of $\tau^{-1}$ on the integral lattice $D^{a,b}_{\cO_{E, \fp}}\{r, s\}$; this is also the reason for the factor $p^{(r+s)t}$ in the definition of the element.

  \begin{theorem}[Wild norm relation]
   \label{thm:wildnorm}
   For $t \ge 1$ we have
   \[ \operatorname{pr}^{K_{G_p}[\fm, p^{t+1}]}_{K_{G_p}[\fm, p^t]}\left(\cZ^{[a,b,r,s]}_{\mot, \fm, p^{t+1}}\right) = \mathcal{U}_p' \cdot \cZ^{[a,b,r,s]}_{\mot, \fm, p^{t}}, \]
   and similarly,
   \[ \operatorname{norm}^{E[\fm p^{t+1}]}_{E[\fm p^t]}\left(\Xi^{[a,b,r,s]}_{\mot, \fm, p^{t+1}}\right) = \sigma_p^{-1} \mathcal{U}_p' \cdot \Xi^{[a,b,r,s]}_{\mot, \fm, p^{t}}.\]
  \end{theorem}

  \begin{note}
   Here $\sigma_p$ is the image of $p^{-1} \in (E \otimes \Qp)^\times$ under the global Artin map, i.e.~the unique element of $\Gal(E[\fm p^t] / E[p^t])$ mapping to the arithmetic Frobenius at $p$ in $\Gal(E[\fm] / E)$.
  \end{note}

  \begin{proof}
   This is a consequence of the general machinery developed in the paper \cite{loeffler-spherical}, which proves a general norm-compatibility statement for elements defined by means of a ``pushforward map of Cartesian cohomology functors'' in the sense of \S 2.3 of \emph{op.cit.}, which is a formalism designed specifically for applications to the cohomology of Shimura varieties and other symmetric spaces.

   More precisely, we take the groups $G$ and $H$ of \emph{op.cit.} to be the $\Qp$-points of the groups $G$ and $H$ of the present paper; then the motivic cohomology groups of the Shimura varieties for $G$ and $H$, and the pushforward maps $\iota_{U, \star}^{[a,b,r,s]}$ between them, described in \S \ref{sect:UEconstruct1} (for varying levels $U$), satisfy the axioms for a pushforward map of the required type. (Compare the case of \'etale cohomology treated in \cite[\S 3.4]{loeffler-spherical}).

   So we may apply the machinery of \S 4 of \emph{op.cit.}, with the parabolic subgroups $Q_G$ and $Q_H$ taken to be the Borel subgroups $B_G$ and $B_H$, and open-orbit representative $u$ taken be the one denoted by the same letter in \cref{lem:openorbit} above. Then the first assertion of the theorem is exactly Proposition 4.5.2 of \emph{op.cit.}; and the second assertion of the theorem follows from the first using \eqref{eq:component}.
  \end{proof}

  \begin{remark}
   Since the operator $\mathcal{U}_p'$ is invertible in the Hecke algebra of level $\Ih_p$, this shows that the classes $\sigma_p^t (\mathcal{U}_p')^{-t} \Xi^{[a,b,r,s]}_{\mot, \fm, p^{t}}$ for varying $t$ and $\fm$ form a ``motivic Euler system'' over all the abelian extensions $E[\fm p^t]$, for $\fm \in \cR$ and $t \ge 1$. However, these classes typically will not have bounded denominators with respect to $t$ in the \'etale realisation, as will become clear from the analysis below.
  \end{remark}

  As noted above, these classes extend naturally to the canonical integral model of $Y_G(K_G[\fm, p^t])$ over $\cO[S^{-1}, \Nm(\fm)^{-1}]$, which we denote by $\cY_{p^t}$. Their \'etale realisations are also integral in another, separate sense: namely, they arise from an integral lattice in the coefficient sheaf, as we now explain. We suppose $\delta_S$ lies in $\cI(G_S / K_{G, S}, \cO_{E, (\fp)})$; and we choose an integer $c > 1$ coprime to $6pS$.

  \begin{theorem}[Wild norm relation, integral \'etale form]
   \label{thm:wildnormint}
   There exists a collection of elements
   \[ {}_c\cZ^{[a,b,r,s]}_{\et, \fm, p^{t}} \in H^3_{\et}\left(\cY_{p^t}, \sD^{a,b}_{\cO_{E, \fp}}\{r, s\}(2)\right) \]
   for all $t \ge 0$ and $\fm \in \cR$ coprime to $c$, such that:
   \begin{enumerate}[(a)]
    \item the image of $z_t$ after inverting $p$ and restricting to the generic fibre is $(c^2 - c^{-n}\langle c \rangle) \cZ^{[a,b,r,s]}_{\et, \fm, p^{t}}$.
    \item For $t \ge 1$ we have the norm relation $\pr_{\cY_t}^{\cY_{t+1}}\left( z_{t+1} \right) = \mathcal{U}_p' \cdot z_t$ (exactly, not just modulo torsion).
   \end{enumerate}
  \end{theorem}

  \begin{proof}
   The integrality of these classes follows by the same argument as \cref{prop:intetale}, with a slight modification: we now need to consider $\xi = \ch(g K_G)$ where $g$ is not a unit at $p$, so the pushforward $g_\star: Y_G(g K_G g^{-1}) \to Y_G(K_G)$ may not respect the integral lattice $\sD^{a,b}_{\cO_{E, \fp}}$. However, we are taking $g_p$ to be a unit multiple of $\tau^t$, and the denominator of $(\tau^t)_\star$ (which corresponds to the action of $\tau^{-t}$ on $D^{a,b}$) is bounded by $p^{(r+s)t}$, which is exactly the normalising factor appearing in the definition of the classes. The fact that these classes are norm-compatible again follows from the norm-compatibility machine developed in \cite{loeffler-spherical}, applied to the integral \'etale cohomology of the two Shimura varieties, rather than motivic cohomology as in Theorem \ref{thm:wildnorm}.
  \end{proof}

  Note that the groups $H^3_{\et}\left(\cY_t, \sD^{a,b}_{\cO_{E, \fp}}\{r, s\}(2)\right)$ are finitely-generated over $\cO_{E, \fp}$ (this is an advantage of working with the integral model $\cY_t$). In particular, the operator $e_p' = \lim_{k \to \infty} \left(\mathcal{U}'_p\right)^{k!}$ is defined on these spaces, and acts as an idempotent. So we can define a class
  \begin{equation}
   \label{eq:defZinfty}
   {}_c\cZ^{[a,b,r,s]}_{\et, \fm, p^\infty} = \left( (\mathcal{U}'_p)^{-t} e_p' \cdot {}_c\cZ^{[a,b,r,s]}_{\et, \fm, p^{t}}\right)_{t \ge 1} \in e'_p \cdot H^3_{\et, \Iw}\left(\cY_\infty, \sD^{a,b}_{\cO_{E, \fp}}\{r, s\}(2)\right),
  \end{equation}
  where the right-hand side is the ``Iwasawa cohomology''
  \[ H^i_{\et, \Iw}\left(\cY_\infty, \sD^{a,b}_{\cO_{E, \fp}}\{r, s\}(2)\right) \coloneqq \varprojlim_t H^i_{\et}\left(\cY_t, \sD^{a,b}_{\cO_{E, \fp}}\{r, s\}(2)\right).\]
  Similarly, we have a version of this for the $\Xi$ classes (where we preserve only the ``abelian part'' of the level tower at $p$): if $R$ denotes the ring $\cO[1/S, 1/\Nm(\fm)]$, and $R_{\fm p^t}$ its integral closure in $E[\fm p^t]$, then we have a class
  \[
   {}_c\Xi^{[a,b,r,s]}_{\et, \fm, p^\infty} \in e'_p \cdot H^3_{\et, \Iw}\left(\cY_{\Ih} \times_R R_{\fm p^\infty}, \sD^{a,b}_{\cO_{E, \fp}}\{r, s\}(2)\right),
  \]
  where $\cY_{\Ih}$ is the $R$-model of $Y_{\Ih}$.

  \begin{remark} \
   \begin{enumerate}
    \item It is natural to ask how the classes $\Xi^{[a,b,r,s]}_{\mot, \fm, p^{t}}$ for $t \ge 1$ (living at Iwahori level) are related to the classes $\Xi^{[a,b,r,s]}_{\mot, \fm}$ of the previous section (which live at prime-to-$p$ level). Using \cref{cor:absnorm1}, it is clear that the pushforward of $\Xi^{[a,b,r,s]}_{\mot, \fm, p^1}$ along $Y_{\Ih} \otimes E[p\fm] \to Y_G[1] \otimes E[\fm]$ is given by $\mathcal{Q}_p \cdot \Xi^{[a,b,r,s]}_{\mot, \fm}$ where $\mathcal{Q}_p$ is some (computable) Hecke operator. Similarly, one can compute Hecke operators relating $\Xi^{[a,b,r,s]}_{\mot, \fm}$ to the projections of $\Xi^{[a,b,r,s]}_{\mot, \fm, p^1}$ to $U'_p$-eigenspaces, much as in \cite[\S 5.7]{KLZ1b}.

    \item For $p = \fp\bar{\fp}$ split in $E$, we can similarly define a family of classes ${}_c\Xi^{[a,b,r,s]}_{\et, \fm, \fp^\infty}$ over the tower of ray class fields modulo $\fm \fp^\infty$, which only requires us to impose ordinarity at $\fp$ (rather than at $p$, which is a stronger condition). The same also holds with $\fp$ and $\bar{\fp}$ interchanged. These results can be obtained in the same way as above, simply replacing the parabolic subgroup $B_G \subset G$ with one of the two non-minimal proper parabolics in $G_{/\Qp}$ and running the machinery of \cite{loeffler-spherical}.\qedhere
   \end{enumerate}
  \end{remark}

\section{Moment maps and twist-compatibility}
\label{sect:moment}
%%%%%%%%%%%%%%%%%%%%%%%%%%

 \subsection{Moment maps for $G$}

  Fix an arbitrary subgroup $K^{(p)}_G \subset G(\Af^{(p)})$ unramified outside $\Sigma$, and write $K_{G}(p^n) = K_G^{(p)}\times K_{G_p}(p^n)$. We assume that $K_{G}(p^t)$ is sufficiently small for all $t \ge 1$. Let $a,b,r,s$ be integers with $a,b \ge 0$ (we do not need to assume $0 \le r \le a, 0 \le s \le b$ at this point).

  \begin{proposition}
   Let $d^{a,b}\{r, s\}$ be the standard highest-weight vector in $D^{a,b}_{\cO_{E, \fp}}\{r, s\}$; and let $d^{a,b}_t\{r, s\}$ be its reduction modulo $\fp^t$. Then the vector $d^{a,b}_t\{r, s\}$ is stable under $K_{G_p}(p^t)$.
  \end{proposition}

  \begin{proof}
   This is clear since the image of $K_{G_p}(p^t)$ modulo $p^t$ is $N_G(\ZZ/p^t)$, which acts trivially on the highest-weight vector by definition.
  \end{proof}

  It follows that $d^{a,b}_t\{r, s\}$ defines a class in $H^0_{\et}(\cY_t, \sD^{a,b}_{t}\{r, s\})$, where $\sD^{a,b}_{t}$ is the mod $\fp^t$ coefficient sheaf, and $\cY_t$ is the smooth model of $Y_G(K_G(p^t))$ over $\cO[1/\Sigma]$ (where $\Sigma$ some finite set of primes which is sufficiently large, but finite and independent of $t$). Cup-product with $d^{a,b}\{r, s\}$ therefore defines a map
  \[ H^3_{\et}(\cY_t, \cO_{E, \fp}(2)) \to H^3_{\et}(\cY_t, \sD^{a,b}_t \{r, s\}(2)) \]
  for each $t \ge 1$, and hence a map
  \[
    \mom^{[a,b,r,s]}_{G, t}: H^3_{\et, \Iw}(\cY_\infty, \cO_{E, \fp}(2))
    \to H^3_{\et}(\cY_t, \sD^{a,b}_{\cO_{E, \fp}} \{r, s\}(2)),
  \]
  mapping an element $(x_T)_{T \ge 1}$ to the element
  \[
   \left(\pr^T_t( x_T \cup d^{a,b}_T\{r, s\})\right)_{T \ge t} \in \varprojlim_{T\ge t} H^3_{\et}(\cY_t, \sD^{a,b}_{T} \{r, s\}(2)) = H^3_{\et}(\cY_t, \sD^{a,b}_{\cO_{E, \fp}} \{r, s\}(2)).
  \]

  Note that these maps are compatible with the action of the Hecke operator $\mathcal{U}_p'$, since $\tau^{-1}$ acts trivially on the highest-weight vector $d^{a,b}$.

 \subsection{Twist-compatibility for $\cZ$'s}

  Now let us suppose $\delta_S$ is some choice of local data at $S$ which lies in $\cI(G_S / K_{G, S}, \cO_{E, (\fp)})$, as in \cref{sect:wildnorm}.

  \begin{theorem}
  \label{thm:twistcompat}
   Let $\fm \in \cR$ be coprime to $c$. There exists an element
   \[ {}_c \cZ_{\fm p^\infty}(\delta_S) \in H^3_{\et, \Iw}(\cY_{\fm p^\infty}, \cO_{E, \fp}(2)) \]
   with the following interpolating property: for all integers $t \ge 1$, $0 \le r \le a$ and $0 \le s \le b$, we have
   \[ \mom^{[a,b,r,s]}_{G, t}\left({}_c \cZ_{\fm p^\infty}\right) = \mathcal{U}_p^{-t} e'_{\ord} \cdot {}_c \cZ^{[a,b,r,s]}_{\et, \fm, p^{t}}.\]
  \end{theorem}

  \begin{proof}
   We shall define ${}_c \cZ_{\fm p^\infty}$ to be the class ${}_c\cZ^{[0,0,0,0]}_{\et, \fm, p^\infty}$ of \eqref{eq:defZinfty}. So we need to show that
   \[
    \mom^{[a,b,r,s]}_{G, t}\left({}_c\cZ^{[0,0,0,0]}_{\et, \fm, p^\infty}\right)=\mathcal{U}_p^{-t} e'_{\ord} \cdot {}_c\cZ^{[a,b,r,s]}_{\et, \fm, p^{t}}.
   \]
   This is true by construction for $(a,b,r,s) = (0,0,0,0)$; our aim is to show that this holds for all possible values of $(a,b,r,s)$.

   If we reduce the coefficients modulo $p^T$ on both sides, for some $T \ge t$, then the equality to be proved is
   \[ \pr^T_t\left( \mathcal{U}_p^{-T} e'_{\ord} {}_c\cZ^{[0,0,0,0]}_{\et, \fm, p^T} \cup d^{a,b}_T\{r, s\}\right) = \mathcal{U}_p^{-t} e'_{\ord} \cdot {}_c\cZ^{[a,b,r,s]}_{\et, \fm, p^{t}}.\]
   Since the classes on the right are norm-compatible in $t$ (integrally), we can reduce to the case $T = t$, so it will suffice to prove that
   \[  {}_c\cZ^{[0,0,0,0]}_{\et, \fm, p^t} \cup d^{a,b}_t\{r, s\} = {}_c\cZ^{[a,b,r,s]}_{\et, \fm, p^{t}} \bmod p^t\]
   as elements of $H^3_{\et}(\cY_t, \sD^{a,b}_t\{r, s\}(2))$.

   Let us write $\tilde{\cY}_t$ for the Shimura variety of level $\tau^r K_G[\fm, p^t] \tau^{-r}$. Then pushforward along $\tau$ gives an isomorphism $\tilde{\cY}_t \to \cY_t$, but the map of sheaves on $\tilde\cY_t$,
   \[ \sD^{a,b}_t\{r, s\} \to \tau^*\left(\sD^{a,b}_t\{r, s\}\right), \]
   corresponds to the action of $\tau^{-t}$ on $\sD_t^{a, b}$, which factors through projection to the highest-weight vector.

   Now, both ${}_c\cZ^{[0,0,0,0]}_{\et, \fm, p^t} \cup d^{a,b}_t\{r, s\}$ and${}_c\cZ^{[a,b,r,s]}_{\et, \fm, p^{t}}$ are in the image of pushforward along $\tilde\cY_t \to \cY_t$: they are the images, respectively, of
   \begin{gather}
   (u_\star \circ \iota_{gU, \star})\left( {}_c\Eis_{\et, \phi[\fm p^t]}^{0}\right) \cup d^{a,b}_t\{r, s\} \qquad\text{and}\qquad
   u_\star \left(\iota^{a,b,q,r}_{gU, \star}\left( {}_c\Eis_{\et, \phi[\fm p^t]}^{n}\right)\right).
   \end{gather}
   The Eisenstein series in the latter class, of weight $n = a + b - r -s$, is congruent modulo $p^t$ (indeed modulo $p^{2t}$) to the cup-product of ${}_c\Eis_{\et, \phi[\fm p^t]}^{0}$ with the highest-weight vector of $\sH^n \bmod p^t$. This highest-weight vector maps to $\br^{[a,b,r,s]} \in D^{a, b}$, so the latter of our two classes on $\tilde\cY_t$ can be written as
   \[ (u_\star \circ \iota_{gU, \star})\left( {}_c\Eis_{\et, \phi[\fm p^t]}^{0}\right) \cup u_\star \br^{[a,b,r,s]}. \]

   Since the classes $u_\star \br^{[a,b,r,s]} = u^{-1} \cdot \br^{[a,b,r,s]}$ and $d^{a,b}_t\{r, s\}$ have the same image in the highest-weight quotient by \cref{prop:brabrs}, they have the same image on $\cY_t$, and the proof is complete.
  \end{proof}

 \subsection{Twist-compatibility for $\Xi$'s}

  Now let $(a,b)$ be given integers $\ge 0$. The same construction as above gives maps
  \[ \mom^{[r, s]}_t:  H^3_{\et, \Iw}\left(\cY_{\Ih} \times_R R_{\fm p^\infty}, \sD^{a,b}_{\cO_{E, \fp}}(2)\right) \to H^3_{\et, \Iw}\left(\cY_{\Ih} \times_R R_{\fm p^t}, \sD^{a,b}_{\cO_{E, \fp}}\{r, s\}(2)\right)\]
  for any $r, s \in \ZZ$ and $t \ge 1$.

  \begin{corollary}
   \label{cor:twistcompatXi}
   Under the same hypotheses as the previous theorem, for any integers $a, b \ge 0$, there is a class
   \[ {}_c \Xi_{\fm p^\infty}^{[a,b]} \in H^3_{\et, \Iw}(\cY_{\Ih} \times_R R_{\fm p^\infty}, \sD^{a,b}_{\cO_{E, \fp}}(2)),\]
   such that for all $(r, s, t)$ with $0 \le r \le a, 0 \le s \le b$, $t \ge 1$, we have
   \[ \mom^{[r, s]}_t\left({}_c \Xi_{\fm p^\infty}^{[a,b]}\right)
   = \sigma_p^t \mathcal{U}_p^{-t} e'_{\ord} \cdot {}_c \Xi^{[a,b,r,s]}_{\et, \fm, p^{t}}.\]
  \end{corollary}

  \begin{proof}
   Immediate from the previous theorem.
  \end{proof}

 \subsection{Cohomological triviality}

  \begin{lemma}
   We have
   \[ \varprojlim_t H^0\left(R[\fm p^t], H^3_{\et}(Y_{\Ih, \overline{\QQ}},  \sD^{a,b}_{\cO_{E, \fp}}(2))\right) = 0.\]
  \end{lemma}

  \begin{proof}
   This follows from the fact that $H^3_{\et}(Y_{\Ih, \overline{\QQ}},  \sD^{a,b}_{\cO_{E, \fp}}(2))$ is a finitely-generated $\cO_{E, \fp}$-module, and $E[\fm p^\infty] / E$ is a positive-dimensional $p$-adic Lie extension.
  \end{proof}

  It follows that there is a map
  \[
   H^3_{\et, \Iw}\left(\cY_{\Ih} \times_R R[\fm p^\infty],  \sD^{a,b}_{\cO_{E, \fp}}(2)\right) \to H^1_{\Iw}\left(R[\fm p^\infty], H^2_{\et}(Y_{\Ih, \overline{\QQ}},  \sD^{a,b}_{\cO_{E, \fp}}(2))\right),
  \]
  and we may regard $ {}_c \Xi_{\fm p^\infty}^{[a,b]}$ as an element of $H^1_{\Iw}\left(R[\fm p^\infty], H^2_{\et}(Y_{\Ih, \overline{\QQ}},  \sD^{a,b}_{\cO_{E, \fp}}(2))\right)$ via this map. We can freely replace $R[\fm p^\infty]$ with $E[\fm p^\infty]$, since any class in the Iwasawa $H^1$ is automatically unramified outside the primes above $p$ (see e.g.~\cite[Corollary B.3.4]{rubin00}).

\section{Mapping to Galois cohomology}
\label{sect:galcoh}

 We now show that the classes ${}_c \Xi_{\fm p^\infty}^{[a, b]}$, projected to a specific Hecke eigenspace, form an ``Euler system'' in the usual sense for the Galois representation associated to a RAECSDC automorphic representation of $\GL_3 / E$. The arguments in this section are very closely parallel to \cite[\S 10.1--10.5]{LSZ17} in the $\operatorname{GSp}_4$ case.

 \begin{remark}
  In this section we won't use the classes ${}_c \cZ_{\fm p^\infty}$. However, these classes can be used to show that the constructions below are compatible with variation in Hida-type families; this will be pursued further elsewhere.
 \end{remark}

 \subsection{Automorphic Galois representations} We recall some results on automorphic Galois representations of $\GL_3 / E$, following \cite{BLGHT11}. Let $\Pi$ be a RAECSDC automorphic representation of $\GL_3 / E$; and for each prime $w$ of $E$ such that $\Pi_w$ is unramified, let $P_w(\Pi, X) \in \CC[X]$ denote the polynomial such that
  \[ P_w(\Pi, \Nm(w)^{-s})^{-1} = L(\Pi_w, s).\]

  \begin{propqed}[{\cite[Theorem 1.2]{BLGHT11}}]
   The coefficients of the polynomials $P_w(\Pi, X)$ lie in a finite extension $F_{\Pi}$ of $E$ independent of $w$; and for each place $\fP \mid p$ of $F_{\Pi}$, there is a 3-dimensional $F_{\Pi, \fP}$-linear representation $V_\fP(\Pi)$ of $\Gal(\bar{E}/E)$, uniquely determined up to semisimplification, with the property that if $w$ is a prime not dividing $p$ for which $\Pi_w$ is unramified, we have
   \[ \det( 1 - X \operatorname{Frob}_w^{-1} : V_\fP(\Pi)) = P_w(\Pi, q X). \qedhere\]
  \end{propqed}

%  We assume henceforth that $\pi$ is globally generic and non-endoscopic, and we let $\Pi = \operatorname{BC}(\pi)$ be the corresponding RAECSDC automorphic representation of $\GL_3 / E$. Comparing the characteristic polynomials of Frobenius elements at good primes, we see that $V_{\fP}(\pi)$ coincides with the Galois representation denoted $r_{p, \iota}(\Pi)$ in , where $\iota$ is any isomorphism $\overline{\QQ}_p \cong \CC$ inducing the prime $\fP$ of $F$.

  \begin{remark}
   If we fix $\Pi$ and let $p$ vary, then \cite[Theorem 2]{xia19} shows that there is a density 1 set of rational primes $p$ such that $V_{\fP}(\Pi)$ is irreducible for all $\fP\mid p$ (and hence unique up to isomorphism).
  \end{remark}

  \subsubsection*{Weights} Since $\Pi$ is regular algebraic, it has a well-defined \emph{weight} at each embedding $\tau: E \into F_{\Pi}$, which is a triple of integers $a_{\tau,1} \ge a_{\tau ,2} \ge a_{\tau ,3}$ (see \cite[\S 1]{BLGHT11}). Since $\Pi^c$ is a twist of $\Pi^\vee$, $a_{\tau, i} + a_{\bar\tau, 4-i}$ is independent of $i$. Thus, up to twisting by an algebraic Gr\"ossencharacter if necessary, we can (and do) assume that the weight of $\Pi$ is $(a+b, b, 0)$ at the identity embedding, and $(a+b, a, 0)$ for the conjugate embedding, for some integers $a, b \ge 0$.

  \begin{proposition}
   The representation $V_{\fP}(\Pi)$ is de Rham at the primes above $p$, and has Hodge numbers\footnote{Negatives of Hodge--Tate weights, so the cyclotomic character has Hodge number $-1$.} $\{0, 1+b, 2 + a + b\}$ at the identity embedding $E \into F_{\fP}$, and $\{0, 1 + a, 2 + a + b\}$ at the conjugate embedding. Moreover, the coefficients of $P_w(\Pi, qX)$ are algebraic integers for all $w$.
  \end{proposition}

  \begin{proof} This follows from part (4) of \cite[Theorem 1.2]{BLGHT11}.
  \end{proof}

  \subsubsection*{Ordinarity} Let $\fp \mid p$ be a prime of $E$ such that $\Pi_{\fp}$ is unramified. Then $V_{\fP}(\Pi)|_{\Gal(\overline{E}_{\fp}/E_{\fp})}$ is crystalline, and the eigenvalues of the linear map $\varphi^{[E_{\fp}: \Qp]}$ on $\mathbf{D}_{\mathrm{cris}}\left(V_{\fP}(\Pi)|_{\Gal(\overline{E}_{\fp}/E_{\fp})}\right)$ are the reciprocal roots of $P_{\fp}(\Pi, qX)$, by \cite[Theorem 1.2(3)]{BLGHT11}.

  \begin{definition}\label{def:ordinary}
   We say $\Pi$ is \emph{ordinary} at the prime $\fp \mid p$ (with respect to the prime $\fP \mid p$ of $F_{\Pi}$) if the polynomial $P_\fp(\Pi, q X)$ has a factor $(1 - \alpha_{\fp} X)$ with $v_{\fP}(\alpha_{\fp}) = 0$.
  \end{definition}

  A standard argument using $p$-adic Hodge theory (see \cite[Lemma 2.2]{BLGHT11}) shows that $\Pi$ is ordinary at $\fp$ if and only if $V_{\fP}(\Pi)$ has a 1-dimensional subspace invariant under $\Gal(\overline{E}_{\fp} / E_{\fp})$ with the Galois group acting on this subspace by an unramified character. If this holds, then dually $V_{\fP}(\Pi)^*$ has a codimension 1 subspace $\mathcal{F}^1_{\fp} V_{\fP}(\Pi)^*$, such that $V_{\fP}(\Pi)^* / \mathcal{F}^1_{\fp}$ is unramified, with arithmetic Frobenius $\operatorname{Frob}_{\fp}$ acting on this quotient by $\alpha_{\fp}$.

  \begin{remark}
   Since $\Pi$ is conjugate self-dual up to a twist, one checks that $V_{\fP}(\Pi)$ has a 1-dimensional invariant subspace at $\fp$ if and only if it has a 2-dimensional invariant subspace at $\bar{\fp}$. So if $\Pi$ is ordinary at all the primes above $p$, then $V_{\fP}(\Pi)$ and its dual preserve a full flag of invariant subspaces at each prime above $p$. (We will not use this fact directly in the present paper, but it may be relevant to future work relating the Euler system constructed here to Selmer groups and $p$-adic $L$-functions.)
  \end{remark}

 \subsection{Realisation via Shimura varieties}

  We add the further assumption that $V_{\fP}(\Pi)$ be irreducible. We now realise this representation in the \'etale cohomology (with compact support) of the infinite-level Shimura variety $Y_G = \varprojlim_KY_G(K)$. Let $\pi$ be the automorphic representation of $G$ corresponding to $\Pi$ (and some choice of $\omega$ such that $(\Pi, \omega)$ is RAECSDC) as in Theorem \ref{thm:autdescent}.

  \begin{theorem}
   The module $H^2_{\et, c}(Y_{G, \bar{\QQ}}, \sV^{a,b}_{E_\fp}) \otimes F_{\fP}$, considered as a representation of $\Gal(\overline{E} / E) \otimes G(\AA_{\mathrm{f}})$, has a direct summand isomorphic to $V_{\fP}(\Pi) \otimes \pif$.
  \end{theorem}

  \begin{proof}
   The computation of the intersection cohomology ${I\!H^2_{\et}}$ of the Baily--Borel compactification of the Picard modular surface is the main result of the volume \cite{ZFPMS}; see in particular \S 4.3 of \cite{rogawski92} for an overview. This computation shows that the intersection cohomology has a direct summand isomorphic to $V_{\fP}(\Pi) \otimes \pif$. There is a natural map from $H^2_{\et, c}$ of the open modular surface to ${I\!H^2_{\et}}$ of the compactification; and the Hecke eigensystems appearing in the kernel and cokernel of this map are associated to non-cuspidal automorphic representations of $\GL_3 / E$. So the map is an isomorphism on the generalised eigenspace for the spherical Hecke algebra associated to $\pif$, which gives the result.
  \end{proof}

  We can thus interpret any $v \in \pif$ as a homomorphism of Galois representations $V_{\fP}(\Pi) \to \varinjlim_{K} H^2_{\et, c}$, or dually as a homomorphism
  \[ \operatorname{pr}_{\Pi, v}:  H^2_{\et}(Y_{G,\bar{\QQ}}, \sD^{a,b}_{E_\fp}(2)) \to V_{\fP}(\Pi)^*, \]
  which we can consider as a ``modular parametrisation'' of the Galois representation $V_{\fP}(\Pi)^*$. This homomorphism factors through projection to $Y_G(K)$ for any level $K$ which fixes $v$.

 \subsection{An Euler system for $V_{\fP}(\Pi)$}

  We now choose the following data:
  \begin{itemize}
  \item A finite $S$ of primes, an open compact $K_{G, S} \subseteq G(\QQ_S)$, and an element $\delta_S \in \cI(G_S / K_{G, S}, \ZZ)$, as in \cref{sect:localdata};
  \item A non-zero vector $v \in \pif$ stable under the group $K_{G, S} \cdot \Ih_p \cdot\, G(\hat{\ZZ}^{S \cup \{p\}})$.
  \item An integer $c$ coprime to $6pS$.
  \end{itemize}

  We suppose that $\Pi$ is ordinary above $p$, and we let $\alpha_p = \prod_{\fp \mid p} \alpha_{\fp}$ where $\alpha_{\fp}$ is as in Definition \ref{def:ordinary}. Then the generalised $\mathcal{U}_p$-eigenspace of $(\pi_p)^{\Ih_p}$ with eigenvalue $\alpha_p$ is 1-dimensional, where $\mathcal{U}_p$ denotes the double-coset operator $[\Ih_p \tau \Ih_p]$ acting on the $\Ih_p$-invariants (this is easily checked from the explicit formulae for Whittaker functions in \S \ref{sect:zeta}; compare \cite[\S 3.5.5]{LSZ17} in the $\operatorname{GSp}_4$ case). We shall choose $v$ to lie in this eigenspace. Then the projection map $\operatorname{pr}_{\Pi, v}$ factors through the $\mathcal{U}_p' = \alpha_p$ eigenspace, and hence through the ordinary idempotent $e'_p$ of \cref{sect:wildnorm}.

  \begin{theorem}[Theorem B]
   \label{thm:ESpi}
   There exists a lattice $T_{\fP}(\Pi)^* \subset V_{\fP}(\Pi)^*$, and a collection of classes
   \[ \mathbf{c}_{\fm}^{\Pi} \in H^1_{\Iw}\left(E[\fm p^\infty], T_{\fP}(\Pi)^*\right) \]
   for all $\fm \in \cR$ coprime to $pc$, with the following properties:
   \begin{enumerate}[(i)]
    \item For $\fm \mid \fn$ we have
    \[
     \operatorname{norm}_{\fm}^{\fn}\left(\mathbf{c}_{\fn}^{\Pi}\right) = \Big(\prod_{w \mid \frac{\fn}{\fm}} P_w(\Pi, \sigma_w^{-1})\Big) \mathbf{c}_{\fm}^{\Pi}.
    \]

    \item For any Gr\"ossencharacter $\eta$ of conductor dividing $\fm p^\infty$ and infinity-type $(s,r)$ [sic], with $0 \le r \le a$ and $0 \le s \le b$, the image of $\mathbf{c}_{\fn}^{\Pi}$ in $H^1\left(E[\fm p^\infty], V_{\fP}(\Pi)^* \otimes \eta^{-1}\right)$ is the \'etale realisation of a motivic cohomology class.

    \item For all $\fp \mid p$, the projection of $\loc_\fp(\mathbf{c}_{\fm}^{\Pi})$ to the group $H^1_{\Iw}\left(E_{\fp} \otimes_E E[\fm p^\infty], V_{\fP}(\Pi)^* / \mathcal{F}^1_{\fp}\right)$ is zero.

   \end{enumerate}
  \end{theorem}

  \begin{proof}
   The choice of $\delta_S$, $K_{G, S}$, and $c$ determines a collection of Iwasawa cohomology classes ${}_c \Xi^{[a,b]}_{\fm p^\infty}$, for all $\fm \in \cR$ coprime to $pc$, taking values in the $e'_p$-ordinary part of $H^2_{\et}(Y_{\Ih, \overline{\QQ}}, \sD^{a,b}_{E_\fp}(2))$. Moreover, these classes all land in a lattice independent of $\fm$.

   The modular parametrisation map $\pr_{\Pi, v}$ sends this lattice in $H^2_{\et}(Y_{\Ih, \overline{\QQ}}, \sD^{a,b}_{E_\fp}(2))$ to a lattice in $V_{\fP}(\Pi)^*$, and we take $T_{\fP}(\Pi)^*$ to be this lattice. Then we may define
   \[
    \mathbf{c}_{\fm}^{\Pi} = \pr_{\Pi, v}\left({}_c \Xi^{[a,b]}_{\fm p^\infty}\right) \in H^1_{\Iw}(E[\fm p^\infty],T_{\fP}(\Pi)^*).
   \]

   We now prove the properties (i)--(iii). Property (i) follows from the tame norm relation \cref{eq:tamenorm2}, but the argument is a little delicate. Since $v \in \pif$ is unramified outside $S \cup \{p\}$, the homomorphism $\pr_{\Pi, v}$ factors through the eigenspace where the Hecke-algebra-valued polynomial $\cP'_w(X)$ acts as $P_w(\Pi, X)$ for all $w \nmid pS$. So \eqref{eq:tamenorm2} shows that the Iwasawa cohomology class
   \[ h =
    \operatorname{norm}_{\fm}^{\fn}\left(\mathbf{c}_{\fn}^{\Pi}\right) - \Big(\prod_{w \mid \frac{\fn}{\fm}} P_w(\Pi, \sigma_w^{-1})\Big) \mathbf{c}_{\fm}^{\Pi}
   \]
   projects to zero in the cohomology of $V_{\fP}(\Pi)^*$ at each finite level in the tower $E[\fm p^\infty]$. Hence its image in the cohomology of the integral lattice $T_{\fP}(\Pi)^*$ lies in the torsion submodule. Since we are assuming $V_{\fP}(\Pi)^*$ to be irreducible, we have $H^0(E[\fm p^\infty], V_{\fP}(\Pi)^*) = 0$, and hence $H^0(E[\fm p^\infty], T_{\fP}(\Pi)^* \otimes \Qp/\Zp)$ is a finite group. So the exponent of this finite group annihilates the torsion submodule of $H^1(E[\fm p^t], T_{\fP}(\Pi)^*)$ for all $t$, and passing to the inverse limit, we deduce that $h$ is annhilated by a finite power of $p$. Since the Iwasawa cohomology of an infinite $p$-adic Lie extension is $p$-torsion-free, we must have $h = 0$, which proves part (i) of the theorem.

   The remaining properties are somewhat simpler. For property (ii), we use the compatibility with moment maps (\cref{cor:twistcompatXi}), and we note that for any $\eta$ of $\infty$-type $(s, r)$ and conductor dividing $\fm p^t$, the twist $V_{\fP}(\Pi)^* \otimes \eta^{-1}$ can be realised as a direct summand of $\operatorname{Ind}_{E[\fm p^t]}^E H^2_{\et}(Y_{\Ih, \overline{\QQ}}, \sD^{a,b}\{r, s\}(2))$, exactly as in the case of Heegner points described in \S 3.4 of \cite{JLZ}. (The switch in ordering of $r$ and $s$ arises because the character $\mu: G \to \operatorname{Res}_{E /\QQ} \GL_1$ corresponds to $\mu_4$, not $\mu_3$, in our parametrisation of algebraic weights.)

   Finally, the local Selmer condition (iii) at the primes above $p$ follows from part (ii), since any class in the image of motivic cohomology must lie in the Bloch--Kato $H^1_{\mathrm{g}}$ subspace at primes above $p$; and this subspace projects to 0 in the cohomology of the quotient (compare \cite[Proposition 11.2.2]{LSZ17}).
  \end{proof}

 \subsection{Concluding remarks}

  \begin{remark}
   The Euler system of Theorem B depends on choices of local data at the primes in $S$: the vector $v \in \pif$ defining the modular parametrisation, and the element $\delta_S \in \cI(G_S / K_{G, S}, \ZZ)$. It should be possible to check that the Euler systems obtained for different choices of these data are proportional to each other, with the proportionality factor being essentially the local zeta integral of \cref{sect:zeta}; compare \cite[\S 6.6]{LZ20-blochkato}.
  \end{remark}

  \begin{remark}
   For part (ii) of Theorem B, we are identifying $\eta$ with a Galois character via the Artin map. Thus $\eta^{-1}$ has Hodge--Tate weights $(-s, -r)$; so the range of $\infty$-types considered in (ii) is precisely the range for which $V_{\fP}(\Pi)^* \otimes \eta^{-1}$ has one Hodge--Tate weight $\le 0$ and two Hodge--Tate weights $\ge 1$ at each of the embeddings $E \into F_{\fP}$. In particular, $V_{\fP}(\Pi)^* \otimes \eta^{-1}$ is ``1-critical'' in the sense of \cite[\S 6]{loefflerzerbes20}, and satisfies the ``rank 1 Panchishkin condition'' of [\emph{op.cit.},~Definition 7.2], with the subspaces $\mathcal{F}^1_{\fp}$ being the Panchishkin submodules. So the above theorem is consistent with the general conjectures formulated in \emph{op.cit.}.

   It is interesting to note that $V_{\fP}(\Pi)^* \otimes \eta^{-1}$ is also 1-critical if $a+1 \le r \le a+b+1$ and $s \le -1$ (or symmetrically if $r \le -1$ and $b+1 \le s \le a+b+1$). We do not know how to construct interesting motivic cohomology classes for twists in this range.
  \end{remark}

  \begin{remark}
   If we assume in addition that $p$ is split in $E$, then we can use the 2-variable Perrin-Riou logarithm map constructed in \cite{loefflerzerbes14} to define two ``motivic $p$-adic $L$-functions'' associated to $\pi$, as measures on the group $\Gal(E[p^\infty] / E)$ (which is isomorphic to the product of $\Zp^2$ and a finite group). More precisely, we have one of these for each prime $\fp_i$ above $p$, interpolating the images of twists of $\operatorname{loc}_{\fp_i}\left(\mathbf{c}_1^{\Pi}\right)$ under the Bloch--Kato logarithm and dual-exponential maps. Forthcoming works by members of our research groups will explore the relation between these ``motivic'' $p$-adic $L$-functions and two other kinds of $p$-adic $L$-function attached to $\pi$: ``analytic'' $p$-adic $L$-functions interpolating critical values of complex $L$-functions, and ``algebraic'' $p$-adic $L$-functions defined as characteristic ideals of appropriate Selmer groups. We hope that it will be possible to formulate an Iwasawa main conjecture in this setting, and prove one divisibility towards this conjecture, by methods similar to those of \cite{LZ20-blochkato}.

   The case of inert $p$ is more mysterious; in this case, $E[p^\infty]$ is a height 2 Lubin--Tate extension at the primes above $p$, and our understanding of local Iwasawa theory for such representations seems insufficient to construct motivic $p$-adic $L$-functions as measures on $\Gal(E[p^\infty] / E)$. However, it may be possible to construct ``signed'' motivic $p$-adic $L$-functions as measures on the cyclotomic Galois group $\Gal(E(\mu_{p^\infty}) / E)$, using the methods of \cite{rockwood20} applied to the induction of $\pi \otimes \eta$ to $\GL_6 / \QQ$.
  \end{remark}

\appendix

\section{Cyclicity of Hecke modules}

 In this section we sketch an explicit proof of the cyclicity theorem \ref{cyc-thm}; our argument is inspired by the proofs of the uniqueness of Whittaker and Shintani functions in the papers \cite{MS-GL,katomurasesugano03} of Murase, Sugano, and Kato.

  \subsection{Hecke algebras and the cyclicity theorem}
 Let $\ell\nmid 2D$ be a prime. Let $K=G(\Zl)$ and $U=H(\Zl)$. These are hyperspecial maximal compacts of $G(\Ql)$ and $H(\Ql)$, respectively.
 There are associated spherical Hecke algebras:
 $$
 \cH_G^0 = C_c(K\bs G(\Ql)/K), \ \ \ \cH_H^0 = C_c(U\bs H(\Ql)/U).
 $$
 The multiplication on these is, of course, just convolution with respect to fixed Haar measures $dg$ and $dh$ on $G(\Ql)$ and $H(\Ql)$, respectively (we can fix the choices
 by requiring that $K$ and $U$ both have volume $1$ under the corresponding measures, but that is not needed below). Both $\cH_G^0$ and $\cH_H^0$ are commutative rings.

 We also consider the space
 $$
 \cH = C_c(U\bs G(\Ql)/K)
 $$
 of smooth, compactly supported functions $f:G(\Ql)\rightarrow \CC$
 that are left $U$-invariant and right $K$-variant.   We endow $\cH$ with the structure of a left $\cH_H^0\otimes\cH_G^0$-module as follows:
 for $\chi\otimes\xi \in \cH_H^0\otimes\cH_G^0$ and $f\in \cH$,
 $$
 (\chi\otimes\xi)*f(x) = \int_{H(\Ql)}\int_{G(\Ql)} \chi(h)f(hxg^{-1})\xi(g) dh\,dg.
 $$
 The main result of this appendix is:
 \begin{theorem}\label{cyc-thm-2}
 As an  $\cH_H^0\otimes\cH^0_G$-module, $\cH$ is cyclic and generated by the characteristic function $f_0 = \ch(K)$ of $K$.
 \end{theorem}

 There are two cases to consider: $\ell$ split in $E$ and $\ell$ inert in $E$. We give details for each case. Our proofs are disappointingly explicit.

  \subsection{The split case} Suppose that $\ell$ splits in $E$: $\ell = w\bar{w}$. Recall that there is a natural isomorphism
 $G(\Ql) \isoarrow \GL_3(\Ql)\times\Qlt$ under which $K$ is identified with $\GL_3(\Zl)\times\Zlt$.
 Similarly, $H(\Ql)$ is identified with $\GL_2(\Ql)\times\Qlt$ and $U$ with $H(\Zl)\times\Zlt$.
 Hereon we will conflate the algebraic groups $H$ and $G$ with their $\Ql$-points. We let $G_0 = \GL_3(\Ql)$ and $K_0 = \GL_3(\Zl)$.

 Under the above identifications, the
 inclusion of $H$ into $G$ becomes
 \begin{equation}\label{A-embed}
 \begin{split}
 H= & \GL_2(\Ql)\times\Qlt\into \GL_3(\Ql)\times\Qlt=G \\
 & (\left(\smallmatrix a & b \\ c & d\endsmallmatrix\right),x) \mapsto (\left(\smallmatrix a & 0 & b \\ 0 & x & 0 \\ c & 0 & d\endsmallmatrix\right),ad-bc).
 \end{split}
 \end{equation}
 Furthermore, these identifications induce ring isomorphisms $\cH_G^0 = \cH^0_{G_0}\otimes\cH_{\GL_1}^0$ and $\cH_H^0 = \cH^0_{\GL_2}\otimes\cH_{\GL_1}$
 as well as a compatible isomorphism $\cH = \cH'\otimes\cH^0_{\GL_1}$ with $\cH' = C_c(U\bs G_0/K_0)$.

   \subsubsection{A simple reduction} Consider $H = \GL_2(\Ql)\times\Qlt$ as a subgroup of $G_0$ via projection to the first factor in the embedding
 \eqref{A-embed}.  Under this embedding we can view $\cH'$ as an $\cH_H^0\otimes\cH_{G_0}^0$-module. To avoid ambiguities, we write
 $\star$ for the convolution action of $\cH_H^0\otimes\cH_{G_0}^0$ on $\cH'$.

 \begin{lemma} If $\cH'$ is a cyclic $\cH_H^0\otimes\cH_{G_0}^0$-module generated by $\ch(K_0)$, then \cref{cyc-thm} is true.
 \end{lemma}

 \begin{proof} Let $f = f_1 \otimes f_2 \in \cH'\otimes\cH_{\GL_1}^0$. Suppose there exist elements $t_i = (a_i,x_i)\in \GL_2(\Ql)\times\Qlt$ and $t_i'\in G_0$,
 $i=1,...,r$, such that
 $$
 f_1 = \sum_i (\ch(Ut_iU)\otimes\ch(K_0t_i'K_0)) \star\ch(K_0).
 $$
 % = \sum_ i \ch(Ut_iU) * \ch(K_0 t_i'K_0).
 %$$
 Let $\chi = \sum_i \ch(Ut_i U) \otimes (\ch(K_0 t_i'K_0)\otimes f_2(\det(a_i)^{-1}(\cdot)) \in \cH_H^0\otimes\cH_G^0$.
 Then it easily follows that $\chi * \ch(K_0) = f_1\otimes f_2$.
 \end{proof}

 So it suffices to prove the cyclicity hypothesis of this lemma. The rest of the proof of \cref{cyc-thm} in the split case will
 therefore focus on proving:

 \begin{proposition}\label{prop-cyc} $\cH'$ is a cyclic $\cH_H^0\otimes\cH_{G_0}^0$-module generated by $\ch(K_0)$.
 \end{proposition}

 For the proof of this proposition it is more convenient to adjust the embedding of $H$ into $G_0$.
 Conjugating by an element of $K_0$ we may view $H$ more naturally as a block diagonal subgroup of $G_0$
 via the embedding that maps $(A,x)\in \GL_2(\Ql)\times\Qlt=H$ to  $\diag(A,x)\in \GL_3(\Ql)=G_0$.

 Our proof of \cref{prop-cyc} begins with two key lemmas.

   \subsubsection{First key lemma}
 For $m=(m_1,m_2,m_3) \in \ZZ^3$, let $t(m) = \diag(\ell^{m_1},\ell^{m_2},\ell^{m_3}) \in G_0$.
 Let
 $$
 \Lambda = \{ (\mu,\lambda) \in \ZZ^3\times\ZZ^3\ : \mu_1 \geq \mu_2, \lambda_1\geq \lambda_2\geq 0 = \lambda_3\}.
 $$
 Let
 $$
 n_ 0 = \left(\smallmatrix 1 & 0 & 1 \\ 0 & 1 & 1 \\ 0 & 0 & 1 \endsmallmatrix\right).
 $$

 \begin{lemma}\label{decomp-lem}
 $G_0= \cup_{(\mu,\lambda)\in \Lambda} U t(\mu) n_0 t(\lambda)K_0$.
 \end{lemma}

 \begin{remark} This decomposition is a disjoint union, but we do not prove this as it is not needed here.
 \end{remark}

 \begin{proof} This essentially comes from \cite{MS-GL}.

 The group $H$ is identified with the Levi subroup of a standard parabolic $P$ of $G_0=\GL_3(\Ql)$ (corresponding to the partition $3 = 2+1$).
 Write $P=HN$ with $N=\{\left(\smallmatrix 1 & & * \\ 0 & 1 & * \\ 0 & 0 & 1\endsmallmatrix\right)\}$ the unipotent radical.
 By Iwasawa decomposition, $G_0 = PK_0 = HNK_0$. As $H = \sqcup_{m\in\ZZ^3, m_1\geq m_2} U t(m) U$ and $U$ normalizes $N$, we have
 $$
 G_0 = \cup_{m\in\ZZ^3, m_1\geq m_2} U t(m) N K.
 $$
 In particular, every double coset $UgK \subset G_0$ is represented by some element of the form
 $$
 t(m)\left(\smallmatrix 1 & 0 & \ell^{-n_1} \\ 0 & 1 & \ell^{-n_2} \\ 0 & 0 & 1 \endsmallmatrix\right), \ \ n_1,n_2\geq 0, \ m_1\geq m-2.
 $$
 We consider such a double coset and representative.

 Suppose $n_2>n_1$. Then
 $$
 \left(\smallmatrix \ell^{m_1} & 0 & 0 \\ 0 & \ell^{m_2} & 0 \\ 0 & 0 & \ell^{m_3}\endsmallmatrix\right)
 \left(\begin{smallmatrix} 1 & 0 & \ell^{-n_1}+\ell^{-n_2} \\ 0 & 1 & \ell^{-n_2} \\ 0 & 0 & 1 \end{smallmatrix}\right) =
 \left(\begin{smallmatrix} 1 & \ell^{m_1-m_2} & 0 \\ 0 & 1 & 0 \\ 0 & 0 & 1 \end{smallmatrix}\right)
 \left(\smallmatrix \ell^{m_1} & 0 & 0 \\ 0 & \ell^{m_2} & 0 \\ 0 & 0 & \ell^{m_3}\endsmallmatrix\right)
 \left(\begin{smallmatrix} 1 & 0 & \ell^{-n_1} \\ 0 & 1 & \ell^{-n_2} \\ 0 & 0 & 1 \end{smallmatrix}\right)
 \left(\begin{smallmatrix} 1 & -1 & 0 \\ 0 & 1 & 0 \\ 0 & 0 & 1 \end{smallmatrix}\right)
 $$
 also belongs to the same double coset. In particular, we can always choose the representative with $n_1\geq n_2\geq 0$.

 Suppose $m_1-n_1< m_2-n_2$, put $n_2' = n_1-m_1+m_2$ (so $n_2<n_2'\leq n_1$). Then
 $$
 \left(\smallmatrix \ell^{m_1} & 0 & 0 \\ 0 & \ell^{m_2} & 0 \\ 0 & 0 & \ell^{m_3}\endsmallmatrix\right)
 \left(\begin{smallmatrix} 1 & 0 & \ell^{-n_1} \\ 0 & 1 & \ell^{-n_2}+\ell^{-n_2'} \\ 0 & 0 & 1 \end{smallmatrix}\right) =
 \left(\begin{smallmatrix} 1 & 0 & 0 \\ 1 & 1 & 0 \\ 0 & 0 & 1 \end{smallmatrix}\right)
 \left(\smallmatrix \ell^{m_1} & 0 & 0 \\ 0 & \ell^{m_2} & 0 \\ 0 & 0 & \ell^{m_3}\endsmallmatrix\right)
 \left(\begin{smallmatrix} 1 & 0 & \ell^{-n_1} \\ 0 & 1 & \ell^{-n_2} \\ 0 & 0 & 1 \end{smallmatrix}\right)
 \left(\begin{smallmatrix} 1 & 0 & 0 \\ -\ell^{m_1-m_2} & 1 & 0 \\ 0 & 0 & 1 \end{smallmatrix}\right)
 $$
 also represents the double coset. So we may choose the representative such that $m_1-n_1\geq m_2-n_2$.

 For such a representative with $n_1\geq n_2$ and $m_1-n_1\geq m_2-n_2$ we have
 $$
 \left(\smallmatrix \ell^{m_1} & 0 & 0 \\ 0 & \ell^{m_2} & 0 \\ 0 & 0 & \ell^{m_3}\endsmallmatrix\right)
 \left(\begin{smallmatrix} 1 & 0 & \ell^{-n_1}+\ell^{-n_2} \\ 0 & 1 & \ell^{-n_2} \\ 0 & 0 & 1 \end{smallmatrix}\right) = t((m_1-n_1,m_2-n_2,m_3) n_0 t(n_1,n_2,0)
 $$
 with $\mu = (m_1-n_1,m_2-n_2,m_3)$ and $\lambda= (n_1,n_2,0)$ such that  $(\mu,\lambda)\in \Lambda$.
 \end{proof}

   \subsubsection{Second key lemma}
 The second key lemma is about the support of certain Hecke operators.

 \begin{lemma}\label{supp-lem} Let $(\mu,\lambda),(\mu',\lambda') \in \Lambda$ with $(\mu,\lambda) \neq (\mu',\lambda')$ . Suppose
 $$
 U t(\mu)^{-1} K_0 t(\lambda)^{-1} K_0 \cap U t(\mu')^{-1} {}^tn_0^{-1} t(\lambda')^{-1} K_0 \neq \emptyset.
 $$
 Then $\lambda_1' \leq \lambda_1$, and if $\lambda_1' = \lambda_2$ then
 $(\mu_1'-\mu_2') + (\lambda_1'-\lambda_2') \leq (\mu_1-\mu_2) + (\lambda_1-\lambda_2)$,  with equality holding only if $(\mu_1'-\mu_2')<(\mu_1-\mu_2)$.
 \end{lemma}

 \begin{proof} Our proof is inspired by the proof of \cite{katomurasesugano03}.
 We proceed by considering the $\ell$-adic valuations of values of various weight functions  in $\ZZ[\GL_3(\Ql)]$.

 Let $I,J \subset \{1,2,3\}$ be two sets of the same cardinality.  Define
 $$
 \Delta_{I,J}(g) = \det((g_{i,j})_{i\in I,j\in J}),
 $$
 and
 $$
 f_{I,J}(g) = \prod_{r=1}^{m}(g_{i_r,j_r}), \ \ \ I=\{i_1,...,i_m\}, J=\{j_1,...,j_m\}, i_1<i_2<\cdots<i_m, j_1<\cdots j_m.
 $$
 Then it is easy to see that
 \begin{equation}\label{A-Delta}
 \Delta_{I,J}(xyz) = \sum_{I',J'} f_{I,I'}(x) \Delta_{I',J'}(y) f_{J',J}(z).
 \end{equation}

 The idea is to chose suitable $I,J$ and evaluate $\Delta = \Delta_{I,J}$
 on $t(\mu')^{-1}{}^tn_0^{-1}t(\lambda')^{-1}$. For the chosen $I,J$, the $\ell$-adic valuation
 of $\Delta(t(\mu')^{-1}{}^tn_0^{-1}t(\lambda')^{-1})$ can be easily expressed in terms
 of $(\mu',\lambda')$.  On the other hand, by hypothesis
 \begin{equation}\label{D-eq1}
 t(\mu')^{-1}{}^tn_0^{-1}t(\lambda')^{-1} = u t(\mu)^{-1} k_1 t(\lambda)^{-1} k_2,
 \end{equation}
 for some $u \in U$ and $k_1,k_2\in K$. We use \eqref{A-Delta} with
 $x=u$, $y=t(\mu)^{-1} k_1 t(\lambda)^{-1}$, and $z=k_2$ to obtain a lower
 bound on the $\ell$-adic valuation in terms of $(\mu,\lambda)$. This yields
 various inequalities that must be satisfied by $(\mu,\lambda)$ and $(\mu',\lambda')$,
 from which we deduce the lemma.

 We apply this first with $I=J=\{1\}$. Then
 $$
 \ord_\ell(\Delta(t(\mu')^{-1}{}^tn_0^{-1}t(\lambda')^{-1})) = -(\mu_1'+\lambda_1').
 $$
 On the other hand, using \eqref{A-Delta} and \eqref{D-eq1}, $\Delta(t(\mu')^{-1}{}^tn_0^{-1}t(\lambda')^{-1})$
 can be expressed as a sum of terms of the form $f_{I,I'}(u)\Delta_{I',J'}(t(\mu)^{-1}k_1t(\lambda)^{-1})f_{J',J}(k_2)$.
 Let $I'=\{i'\}$ and $J'=\{j'\}$. The $\ell$-adic valuation of such a term is at least $-(mu_{i'}+\lambda_{j'})$.
 As $u\in U$ and $I={1}$, $f_{I,I'}(u) \neq 0$ only if $i' \in \{1,2\}$.
 It follows that
 $$
 \ord_\ell(\Delta(t(\mu')^{-1}{}^tn_0^{-1}t(\lambda')^{-1})) \geq \min_{1\leq i'\leq 2, 1\leq j'\leq 3} \{-(\mu_{i'}+\lambda_{j'})\} = -(\mu_1+\lambda_1).
 $$
 Hence,
 \begin{equation}\label{D-eq2}
 \mu_1'+\lambda_1' \leq \mu_1 + \lambda_1.
 \end{equation}

 Taking $I=\{3\}$ and $J = \{1\}$, a similar analysis yields
 \begin{equation}\label{D-eq4}
 \mu_3' + \lambda_1' \leq \mu_3 + \lambda_1.
 \end{equation}
 Taking $I=J=\{1,2\}$ yields
 \begin{equation}\label{D-eq3}
 \mu_1'+\mu_2' + \lambda_1' + \lambda_2' \leq \mu_1+\mu_2 + \lambda_1 + \lambda_2.
 \end{equation}
 Taking $I=\{1,3\}$ and $J=\{1,2\}$ yields
 \begin{equation}\label{D-eq6}
 \mu_1'+\mu_3' + \lambda_1'+\lambda_2' \leq \mu_1+\mu_3 + \lambda_1+\lambda_2.
 \end{equation}
 And taking $I = J = \{1,2,3\}$ (that is, comparing determinants) yields
 \begin{equation}\label{D-eq5}
 \mu_1'+\mu_2'+\mu_3' + \lambda_1'+\lambda_2' = \mu_1 + \mu_2+\mu_3 + \lambda_1+\lambda_2.
 \end{equation}

 Comparing \eqref{D-eq3} and \eqref{D-eq5} shows that
 \begin{equation}\label{D-eq7}
 \mu'_3 \geq \mu_3.
 \end{equation}
 And comparing this with \eqref{D-eq4} yields
 \begin{equation}\label{D-eq8}
 \lambda_1'\leq \lambda_1  \ \text{and that} \ \lambda_1'=\lambda_1 \implies \mu_3'=\mu_3.
 \end{equation}

 Suppose $\lambda'_1 = \lambda_1$. Then $\mu'_3 = \mu_3$ by \eqref{D-eq8}.
 Combining this with \eqref{D-eq2} and \eqref{D-eq5} yields
 \begin{equation}\label{D-eq9}
 (\mu_1'-\mu_2') + (\lambda_1'-\lambda_2') \leq (\mu_1-\mu_2) + (\lambda_1-\lambda_2), \ \ \text{with equality iff $\mu_1'=\mu_1$.}
 \end{equation}

 Supposing further that $(\mu_1'-\mu_2') + (\lambda_1'-\lambda_2') \leq (\mu_1-\mu_2) + (\lambda_1-\lambda_2)$, so $\mu_1'=\mu_1$ by \eqref{D-eq9}.
 It then follows from \eqref{D-eq6} that
 $$
 \lambda_2' \leq \lambda_2,
 $$
 while it then follows from \eqref{D-eq5} that $\mu_2'-\mu_2 = \lambda_2-\lambda_2'-\lambda_2'$.
 In particular, if $\lambda_2' = \lambda_2$, then $\mu_2 = \mu_2'$ and so $(\mu',\lambda') = (\mu,\lambda)$.
 So it must be that $\lambda_2'<\lambda_2$ and hence that $\mu_2'>\mu_2$. The last equality then implies
 that $\mu_1'-\mu_2' = \mu_1-\mu_2' < \mu_1-\mu_2$.
 This completes the proof of the lemma.
 \end{proof}

   \subsubsection{Proof of \cref{prop-cyc}}
 Let $\cH'' = (\cH_H^0\otimes \cH_{G_0}^0)\star \ch(K_0)$. Let $n_1 = {}^tn_0^{-1}$.
 By \cref{decomp-lem},
 $G_0 = \cup_{(\mu,\lambda)\in\Lambda U} t(\mu)^{-1} n_1 t(\lambda)^{-1}K_0.$
 So it suffices to show that for each $(\mu,\lambda)\in \Lambda$,
 \begin{equation}\label{A-incl}
 \ch(U t(\mu)^{-1} n_1 t(\lambda)^{-1} K_0) \in \cH''.
 \end{equation}

 Let $(\mu,\lambda)\in \Lambda$. We define $\tilde\mu = (\mu_1-\mu_2)$ and $\tilde\lambda = (\lambda_1-\lambda_2)$.
 Our proof is by induction on the set $S$ of ordered triples $s(\mu,\lambda) = (\lambda_1,\tilde\mu+\tilde\lambda,\tilde\mu)$
 of non-negative integers. The set $S$ is well-ordered under the lexicographic ordering.

 The base case of the induction is the inclusion \eqref{A-incl} for all $(\mu,\lambda)$ with $s(\mu,\lambda) = (0,0,0)$.
 For such a $(\mu,\lambda)$, $\lambda = (0,0,0)$ and so
 $$
 \ch(U t(\mu)^{-1} n_1 t(\lambda)^{-1} K_0) = \ch(U t(\mu)^{-1} K_0) = \ch(Ut(\mu)U)\star\ch(K_0) \in \cH''.
 $$
 This proves the base case of the induction.

 Suppose $(\mu,\lambda)\in \Lambda$.
 Let $\chi = \ch(Ut(\mu)U) \in \cH_H^0$ and $\xi = \ch(K_0t(\lambda)^{-1}K_0)\in \cH_{G_0}^0$. The support of $\chi\star\xi = (\chi\otimes\xi)\star\ch(K_0)\in \cH''$
 is exactly $Ut(\mu)^{-1}K_0t(\lambda)^{-1}K_0$. Let $\Lambda(\mu,\lambda)\subset \Lambda$ be the set of $(\mu',\lambda')$ such that
 $$
 Ut(\mu)^{-1}Kt(\lambda)^{-1}K\cap Ut(\mu')^{-1}n_1t(\lambda')^{-1} K_0 \neq \emptyset.
 $$
 It follows from \cref{decomp-lem} that $\chi\star\xi$ can be expressed as a sum over the $(\mu',\lambda')\in \Lambda(\mu,\lambda)$
 of scalar multiples of the functions $\ch(U t(\mu')^{-1} n_1 t(\lambda')^{-1} K_0)$. So to show that the particular class $\ch(U t(\mu)^{-1} n_1 t(\lambda)^{-1} K_0)$ is in $\cH''$,
 it suffices to
 show that $\ch(U t(\mu')^{-1} n_1 t(\lambda')^{-1} K_0)\in \cH''$ for all $(\mu',\lambda')\in\Lambda(\mu,\lambda)$ with $(\mu',\lambda')\neq (\mu,\lambda)$.
 But for such a $(\lambda',\mu')$, \cref{supp-lem} implies that
 \begin{equation}\label{D-eq10}
 s(\mu',\lambda') = (\lambda_1',\tilde\mu'+\tilde\lambda',\tilde\mu')< (\lambda_1,\tilde\mu+\tilde\lambda,\tilde\mu) = s(\mu,\lambda)
 \end{equation}
 in the lexicographic ordering.  The induction step follows easily.

  \subsection{The inert case} Suppose that $\ell$ is inert in $E$. Our proof of \cref{cyc-thm-2} in this case follows the same lines as in the split case and is even slightly
 simpler.  As in the split case, we begin by proving two key lemmas, the analogs of Lemmas \ref{decomp-lem} and \ref{supp-lem}.

   \subsubsection{First key lemma} For $m=(m_1,m_2) \in \ZZ^2$ we let $t(m) = \diag(\ell^{m_1},\ell^{m_2}, \ell^{2m_2-m_1})\in T$.
 We let
 $$
 \Lambda = \{(\mu,\lambda) \in \ZZ^2\times\ZZ^2) \ : \mu_1\geq \mu_2, \lambda_1\geq 0 = \lambda_2\}.
 $$
 Using the parametrisation of $N_G(\Ql)$ as $\{ n(x, y): x \in \cO \otimes \Zl, y \in \Zl\}$ given in Lemma \ref{lem:boreldecomp}, for $s \in \ZZ$ we set
 $$
 n_s = n(\ell^s, 0) \in N_G(\Ql). %. \left(\smallmatrix 1 & 1 & 0 \\ 0 & 1 & 1 & \\ 0 & 0 & 1 \endsmallmatrix\right).
 $$

 \begin{lemma}\label{decomp-lem2}
 $G = \cup_{(\mu,\lambda)\in \Lambda)} U t(\mu) n_0 t(\lambda) K$.
 \end{lemma}

 \begin{proof} Let $N_0 = N_G(\Zl)$, and for $r \ge 1$ let $N_r$ be the kernel of reduction mod $\ell^r$ on $N_G(\Zl)$.
 Let
 $$
 w = \left(\smallmatrix 0 & 0 & 1 \\ 0 & 1 & 0 \\ -1 & 0 & 0 \endsmallmatrix\right).
 $$
 This represents the longest element (in this case, the non-trivial) Weyl element.
 Let $\bar{N}_r = wN_rw^{-1}$. Let $T_0 = T(\Zl)$. Then the Iwahori subgroup (with respect
 to the upper-triangular Borel $B$) is just the group $K_B=T_0N_0\bar{N}_1 = T_0 \bar{N}_1 N_0$,
 and the Iwahori decomposition of $K$ is just
 $$
 K = K_B \sqcup K_BwN_0 = T_0 \bar{N}_1 N_0 \sqcup T_0 N_0 w N_0.
 $$
 From this we deduce that
 \begin{equation}\label{Iwahori-ish}
 K = Kw = T_0wN_1 \bar{N}_0\ \sqcup\ T_0 N_0 \bar{N}_0.
 \end{equation}

 Let $T^+ = \{t(m) \ : \ m\in \ZZ^2, m_1\geq m_2\}$. By Iwasawa decomposition, $G = KT^+K$, so by \eqref{Iwahori-ish}
 $$
 G = T_0w N_1 \bar{N}_0 T^+K\ \cup\ T_0 N_0 \bar{N}_0 T^+ K.
 $$
 As $\bar{N}_r T^+ K = T^+ K$ and $T_0, T_0 w \subset U$, it follows that
 \begin{equation}\label{decomp-eq1}
  G = U N_0 T^+ K.
 \end{equation}

 The elements $n(x, 0)$, for $x \in \cO\otimes \Zl$, give coset representatives for $(U \cap N_0) \backslash N_0$. Since may rescale $x$ by elements of $(\cO\otimes \Zl)^\times$ using the commutation relation in Lemma \ref{lem:boreldecomp}, it follows from \eqref{decomp-eq1} that every double coset $UgK$ has a representative
 of the form $n_s t(m)$ with $s\geq 0$ and $m_1\geq m_2$.
 As $n_s t(m)= t(m)n_{s-m_1+m_2}$, it follows that $t(m)n_{s'}$, $s' = \min\{0,s-m_1+m_2\}$ also represents the double coset.
 But $t(m)n_{s'} = t(\mu) n_0 t(\lambda),  \ \mu=(m_1+s',m_2), \lambda = (-s',0)$. That $(\mu,\lambda)\in\Lambda$ follows
 from $s\geq 0$ and the definition of $s'$.
 \end{proof}

   \subsubsection{Second key lemma}

 \begin{lemma}
 Let $(\mu,\lambda),(\mu',\lambda') \in \Lambda$ with $(\mu,\lambda) \neq (\mu',\lambda')$ . Suppose
 $$
 \left(U t(\mu)^{-1} K t(\lambda)^{-1} K\right) \cap \left(U t(\mu')^{-1} \, {}^tn_0^{-1}\, t(\lambda')^{-1} K\right) \neq \emptyset.
 $$
 Then $\lambda_1' \leq \lambda_1$, and if $\lambda_1' = \lambda_1$ then $\mu_1'-\mu_2' < \mu_1-\mu_2$.
 \end{lemma}

 \begin{proof} The proof is much the same as before, exploiting the functions $\Delta_{I,J}$.
 Taking $I = J = \{1\}$ yields
 \begin{equation}\label{D2-eq1}
 \mu_1'+\lambda_1' \leq \mu_1+\lambda_1.
 \end{equation}
 Taking $I={2}$, $J=\{1\}$ yields
 \begin{equation}\label{D2-eq2}
 \mu_2'+\lambda_1' \leq \mu_2+\lambda_1.
 \end{equation}
 Comparing similitude factors gives
 \begin{equation}\label{D2-eq3}
 \mu_2' = \mu_2.
 \end{equation}
 From \eqref{D2-eq2} and \eqref{D2-eq3} we conclude that
 \begin{equation}\label{D2-eq4}
 \lambda_1'\leq \lambda_1.
 \end{equation}
 If $\lambda'_1 = \lambda_1$, then \eqref{D2-eq1} implies that $\mu_1'\leq \mu_1$, from which it follows
 that $\mu_1'-\mu_2' = \mu_1'-\mu_2 \leq \mu_1-\mu_2$ with equality only if $\mu_1'=\mu_1$ (in which case $(\mu',\lambda') = (\mu,\lambda)$).
 \end{proof}

  \subsection{Proof of \cref{cyc-thm-2}}
 The theorem follows easily from induction on the ordered pairs $(\lambda_1,\mu_1-\mu_2)$ of non-negative integers, in exact analogy with the proof
 of \cref{prop-cyc}.

\let\oldbibliography\thebibliography
\renewcommand{\thebibliography}[1]{%
  \oldbibliography{#1}%
  \setlength{\itemsep}{1.5ex}%
}
%\bibliographystyle{../amsalphaurl}
%\bibliography{../references}

\begin{thebibliography}{BLGHT11}

\bibitem[Anc15]{ancona15}
\textsc{G.~Ancona},
  \articlehref{http://dx.doi.org/10.1007/s00229-014-0708-4}{\emph{D{\'e}composition
  de motifs ab{\'e}liens}}, Manuscripta Math. \textbf{146} (2015), no.~3--4,
  307--328. \MR{3312448}

\bibitem[BLGHT11]{BLGHT11}
\textsc{T.~Barnet-Lamb, D.~Geraghty, M.~Harris, and R.~Taylor},
  \articlehref{http://dx.doi.org/10.2977/PRIMS/31}{\emph{A family of
  {C}alabi-{Y}au varieties and potential automorphy {II}}}, Publ. Res. Inst.
  Math. Sci. (Kyoto) \textbf{47} (2011), no.~1, 29--98. \MR{2827723}

\bibitem[CHL11]{clozelharrislabesse}
\textsc{L.~Clozel, M.~Harris, and J.-P. Labesse}, \emph{Construction of
  automorphic {G}alois representations, {I}}, Stabilization of the trace
  formula, {S}himura varieties, and arithmetic applications: {I}. {O}n the
  stabilization of the trace formula, Int. Press, Somerville, MA, 2011,
  pp.~497--527. \MR{2856383}

\bibitem[GPS84]{Gel-PS-unitary}
\textsc{S.~Gelbart and I.~I. Piatetski-Shapiro},
  \articlehref{http://dx.doi.org/10.1007/BFb0073147}{\emph{Automorphic forms
  and {$L$}-functions for the unitary group}}, Lie group representations, {II}
  ({C}ollege {P}ark, {M}d., 1982/1983), Lecture Notes in Math., vol. 1041,
  Springer, Berlin, 1984, pp.~141--184. \MR{748507}

\bibitem[GW09]{goodmanwallach09}
\textsc{R.~Goodman and N.~R. Wallach},
  \articlehref{http://dx.doi.org/10.1007/978-0-387-79852-3}{\emph{Symmetry,
  representations, and invariants}}, Graduate Texts in Mathematics, vol. 255,
  Springer, 2009. \MR{2522486}

\bibitem[Gor92]{gordon92}
\textsc{B.~Gordon}, \emph{Canonical models of {P}icard modular surfaces},
  chapter in \cite{ZFPMS}, 1992, pp.~1--29. \MR{1155224}

\bibitem[GS20]{grahamshah20}
\textsc{A.~Graham and S.~W.~A. Shah},
  \articlehref{http://arxiv.org/abs/2001.07825}{\emph{Anticyclotomic {E}uler
  systems for unitary groups}}, preprint, 2020, \path{arXiv:2001.07825}.

\bibitem[JLZ21]{JLZ}
\textsc{D.~Jetchev, D.~Loeffler, and S.~L. Zerbes},
  \articlehref{http://dx.doi.org/10.1112/plms.12363}{\emph{Heegner points in
  {C}oleman families}}, Proc. Lond. Math. Soc. \textbf{122} (2021), no.~1,
  124--152. \MR{4210260}

\bibitem[Kat04]{kato04}
\textsc{K.~Kato},
  \articlehref{http://smf4.emath.fr/en/Publications/Asterisque/2004/295/html/smf_ast_295_117-290.html}{\emph{{$P$}-adic
  {H}odge theory and values of zeta functions of modular forms}},
  Ast{\'e}risque \textbf{295} (2004), 117--290, Cohomologies $p$-adiques et
  applications arithm{\'e}tiques. III. \MR{2104361}

\bibitem[KMS03]{katomurasesugano03}
\textsc{S.-i. Kato, A.~Murase, and T.~Sugano},
  \articlehref{http://dx.doi.org/10.2748/tmj/1113247445}{\emph{Whittaker--{S}hintani
  functions for orthogonal groups}}, Tohoku Math. J. \textbf{55} (2003), no.~1,
  1--64. \MR{1956080}

\bibitem[KLZ17]{KLZ1b}
\textsc{G.~Kings, D.~Loeffler, and S.~L. Zerbes},
  \articlehref{http://dx.doi.org/10.4310/CJM.2017.v5.n1.a1}{\emph{{R}ankin--{E}isenstein
  classes and explicit reciprocity laws}}, Cambridge J. Math. \textbf{5}
  (2017), no.~1, 1--122. \MR{3637653}

\bibitem[LR92]{ZFPMS}
\textsc{R.~Langlands and D.~Ramakrishnan} (eds.), \emph{The zeta functions of
  {P}icard modular surfaces}, Universit\'{e} de Montr\'{e}al, Centre de
  Recherches Math\'{e}matiques, Montreal, QC, 1992. \MR{1155223}

\bibitem[Lin92]{lin92}
\textsc{Z.~Z. Lin},
  \articlehref{http://dx.doi.org/10.2307/2159384}{\emph{Representations of
  {C}hevalley groups arising from admissible lattices}}, Proc. Amer. Math. Soc.
  \textbf{114} (1992), no.~3, 651--660. \MR{1079702}

\bibitem[Loe20]{loeffler-spherical}
\textsc{D.~Loeffler},
  \articlehref{http://arxiv.org/abs/1909.09997}{\emph{Spherical varieties and
  norm relations in {I}wasawa theory}}, J. Th{\'e}or. Nombres Bordeaux (2020),
  Iwasawa 2019 special issue, to appear, \path{arXiv:1909.09997}.

\bibitem[LPSZ19]{LPSZ1}
\textsc{D.~Loeffler, V.~Pilloni, C.~Skinner, and S.~L. Zerbes},
  \articlehref{http://arxiv.org/abs/1905.08779}{\emph{Higher {H}ida theory and
  {$p$}-adic {$L$}-functions for {$\operatorname{GSp}(4)$}}}, to appear in Duke
  Math. J., 2019, \path{arXiv:1905.08779}.

\bibitem[LSZ17]{LSZ17}
\textsc{D.~Loeffler, C.~Skinner, and S.~L. Zerbes},
  \articlehref{http://arxiv.org/abs/1706.00201}{\emph{Euler systems for
  {$\operatorname{GSp}(4)$}}}, J. Eur. Math. Soc. (2017), to appear,
  \path{arXiv:1706.00201}.

\bibitem[LZ14]{loefflerzerbes14}
\textsc{D.~Loeffler and S.~L. Zerbes},
  \articlehref{http://dx.doi.org/10.1142/S1793042114500699}{\emph{Iwasawa
  theory and {$p$}-adic {$L$}-functions over {$\mathbf{Z}_p^2$}-extensions}},
  Int. J. Number Theory \textbf{10} (2014), no.~8, 2045--2096. \MR{3273476}

\bibitem[LZ20a]{loefflerzerbes20}
\bysame, \articlehref{http://dx.doi.org/10.2969/aspm/08610001}{\emph{Euler
  systems with local conditions}}, Development of Iwasawa Theory -- the
  Centennial of K.~Iwasawa's Birth, Adv. Stud. Pure Math., no.~86, Math. Soc.
  Japan, 2020.

\bibitem[LZ20b]{LZ20-blochkato}
\bysame, \articlehref{http://arxiv.org/abs/2003.05960}{\emph{On the
  {B}loch--{K}ato conjecture for {$\operatorname{GSp}(4)$}}}, preprint, 2020,
  \path{arXiv:2003.05960}.

\bibitem[Mok15]{mok15}
\textsc{C.~P. Mok},
  \articlehref{http://dx.doi.org/10.1090/memo/1108}{\emph{Endoscopic
  classification of representations of quasi-split unitary groups}}, vol. 235,
  Mem. Amer. Math. Soc., no. 1108, 2015. \MR{3338302}

\bibitem[MS96]{MS-GL}
\textsc{A.~Murase and T.~Sugano},
  \articlehref{http://dx.doi.org/10.2748/tmj/1178225376}{\emph{Shintani
  functions and automorphic {$L$}-functions for {$\operatorname{GL}(n)$}}},
  Tohoku Math. J. (2) \textbf{48} (1996), no.~2, 165--202. \MR{1387815}

\bibitem[Pin90]{pink90}
\textsc{R.~Pink}, \emph{Arithmetical compactification of mixed {S}himura
  varieties}, Bonner Mathematische Schriften, vol. 209, Universit\"at Bonn,
  1990. \MR{1128753}

\bibitem[PS18]{pollackshah18}
\textsc{A.~Pollack and S.~Shah},
  \articlehref{http://arxiv.org/abs/1801.07383}{\emph{A class number formula
  for {P}icard modular surfaces}}, preprint, 2018, \path{arXiv:1801.07383}.

\bibitem[Roc19]{rockwood20}
\textsc{R.~Rockwood},
  \articlehref{http://arxiv.org/abs/1912.09375}{\emph{Plus/minus {$p$}-adic
  {$L$}-functions for {$\operatorname{GL}_{2n}$}}}, preprint, 2019,
  \path{arXiv:1912.09375}.

\bibitem[Rog92]{rogawski92}
\textsc{J.~Rogawski}, \emph{Analytic expression for the number of points mod
  {$p$}}, chapter in \cite{ZFPMS}, 1992, pp.~65--109. \MR{1155227}

\bibitem[Rub00]{rubin00}
\textsc{K.~Rubin}, \emph{Euler systems}, Ann. of Math. Stud., vol. 147,
  Princeton Univ. Press, 2000. \MR{1749177}

\bibitem[Sak13]{sakellaridis13}
\textsc{Y.~Sakellaridis},
  \articlehref{http://dx.doi.org/10.1353/ajm.2013.0046}{\emph{Spherical
  functions on spherical varieties}}, Amer. J. Math. \textbf{135} (2013),
  no.~5, 1291--1381. \MR{3117308}

\bibitem[Shi76]{shintani76}
\textsc{T.~Shintani},
  \articlehref{http://dx.doi.org/10.3792/pja/1195518347}{\emph{On an explicit
  formula for class-{$1$} ``{W}hittaker functions'' on {$GL_{n}$} over
  {$P$}-adic fields}}, Proc. Japan Acad. \textbf{52} (1976), no.~4, 180--182.
  \MR{407208}

\bibitem[Tor19]{torzewski18}
\textsc{A.~Torzewski},
  \articlehref{http://dx.doi.org/10.1007/s00229-019-01150-9}{\emph{Functoriality
  of motivic lifts of the canonical construction}}, Manuscripta Math.
  \textbf{163} (2019), no.~1--2, 27--56.

\bibitem[Xia19]{xia19}
\textsc{Y.~Xia},
  \articlehref{http://dx.doi.org/10.1007/s00208-018-1786-5}{\emph{Irreducibility
  of automorphic {G}alois representations of low dimensions}}, Math. Ann.
  \textbf{374} (2019), no.~3-4, 1953--1986. \MR{3985128}

\end{thebibliography}
\providecommand{\bysame}{\leavevmode\hbox to3em{\hrulefill}\thinspace}
\providecommand{\MR}[1]{}
\renewcommand{\MR}[1]{%
 MR \href{http://www.ams.org/mathscinet-getitem?mr=#1}{#1}.
}
\providecommand{\href}[2]{#2}
\newcommand{\articlehref}[2]{\href{#1}{#2}}

\end{document}